\documentclass[11pt,a4paper]{amsart}
\usepackage{graphicx,multirow,array,amsmath,amssymb,color,enumitem,subfig}

\newtheorem{theorem}{Theorem}[section]
\newtheorem{lemma}[theorem]{Lemma}

\newtheorem{corollary}[theorem]{Corollary}

\newtheorem{question}[theorem]{Question}

\newcommand{\ty}{\nabla\mathrm{Y}}
\newcommand{\yt}{\mathrm{Y}\nabla}
\newcommand{\DD}{\mathcal{D}_4}
\newcommand{\F}{\mathcal{F}}
\newcommand{\PP}{\mathcal{P}}
\newcommand{\R}{\mathbb{R}}

\begin{document}
	
\title{Dips at small sizes for topological graph obstruction sets}
	
\author[H.\ Kim]{Hyoungjun Kim}
\address{Institute of Data Science, Korea University, Seoul 02841, Korea}
\email{kimhjun@korea.ac.kr}
\author[T.W.\ Mattman]{Thomas W.\ Mattman}
\address{Department of Mathematics and Statistics,
California State University, Chico,
Chico, CA 95929-0525}
\email{TMattman@CSUChico.edu}

\begin{abstract}
The Graph Minor Theorem of Robertson and Seymour implies 
a finite set of obstructions for
any minor closed graph property. 
We show that there are only three obstructions
to knotless embedding of size 23, which is far fewer than the 92 of size 22 and the hundreds known to exist at larger sizes.
We describe several other topological properties whose obstruction set
demonstrates a similar dip at small size. For
order ten graphs, we classify the 35 obstructions to knotless embedding and the 49 maximal knotless graphs.
\end{abstract}

\thanks{The first author(Hyoungjun Kim) was supported by the National Research
Foundation of Korea (NRF) grant funded by the Korea government Ministry of Science 
and ICT(NRF-2021R1C1C1012299 and NRF-2022M3J6A1063595 ).}

\maketitle

\section{Introduction}
	
The Graph Minor Theorem of Robertson and Seymour~\cite{RS} implies that 
any minor closed graph property $\PP$ is characterized by a finite
set of obstructions. For example, planarity is
determined by $K_5$ and $K_{3,3}$ \cite{K,W} while 
linkless embeddability has
seven obstructions, known as the Petersen family~\cite{RST}.
However, Robertson and Seymour's proof is highly non-constructive and it remains frustratingly difficult to identify the obstruction set, even for a simple property such as apex (see~\cite{JK}).	
Although we know that the obstruction set for a property $\PP$ is finite, in practice it is often difficult to establish any useful bounds on its size. 
	
In the absence of concrete bounds,
information about the shape or distribution of an 
obstruction set would be welcome. Given that the number of obstructions is finite, one might anticipate a unimodal distribution with a central maximum and numbers tailing off to either side.
Indeed, many of the known obstruction sets do appear to follow this
pattern. In \cite[Table 2]{MW}, the authors present a listing of more than 17 thousand obstructions for torus embeddings. Although, the list is likely incomplete, it does appear to follow a normal
distribution both with respect to graph {\em order} (number of vertices)
and graph {\em size} (number of edges), see 
Figure~\ref{fig:TorObs}.
	
\begin{figure}[htb]
\centering
\subfloat[By order]{\includegraphics[width=0.4\textwidth]{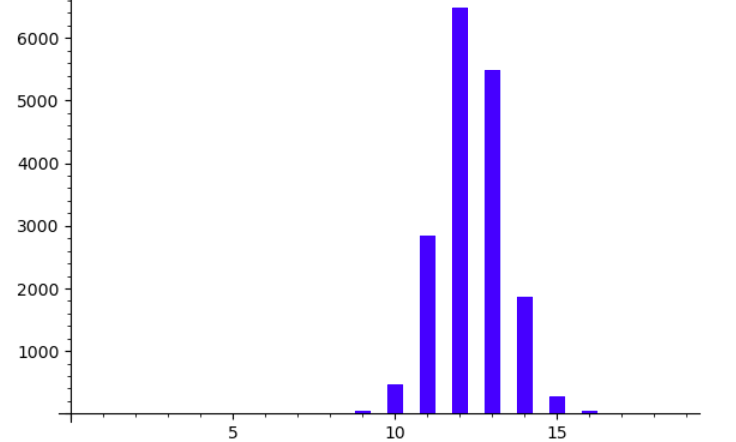}}
\hfill
\subfloat[By size]{\includegraphics[width=0.4\textwidth]{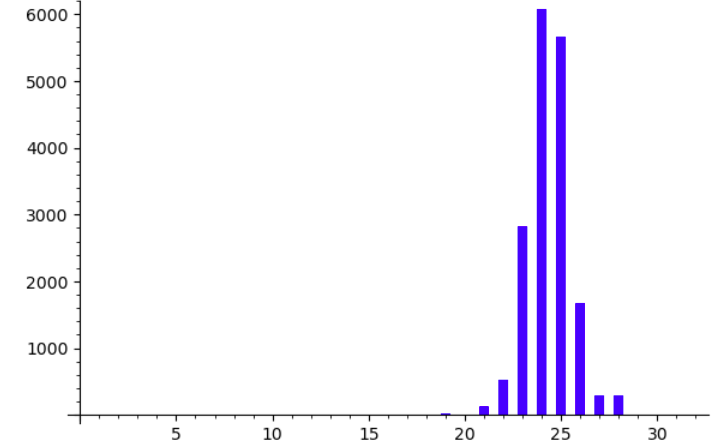}}
\caption{Distribution of torus embedding obstructions}
\label{fig:TorObs}
\end{figure}

\begin{table}[ht]
\centering
\begin{tabular}{l|ccccccccccc}
Size & 18 & 19 & 20 & 21 & 22 & 23 & $\cdots$ & 28 & 29 & 30 & 31 \\ \hline
Count & 6 & 19 & 8 & 123 & 517 & 2821 & $\cdots$ & 299 & 8 & 4 & 1 \end{tabular}
\caption{Count of torus embedding obstructions by size.}
\label{tab:TorSiz}
\end{table}
	
However, closer inspection (see Table~\ref{tab:TorSiz}) shows
that there is a {\em dip}, or local minimum, in the number of obstructions at size twenty.
We will say that the dip occurs at
{\em small size} meaning it is near the end of the left tail
of the size distribution.

In this paper, we will see that the knotless embedding property likewise
has a dip at size 23. 
A {\em knotless embedding} of a graph is an embedding in $\R^3$ such
that each cycle is a trivial knot. 
Having noticed this dip we investigated what
other topological properties have a dip, or even {\em gap} (a size,
or range of sizes, for which there is no obstruction), at small sizes.
We report on what we found in the next section. In a word, the most
prominent dips and gaps
seem to trace back to that perennial counterexample, the 
Petersen graph. 
	
In Section 3, we prove the following.
\begin{theorem} \label{thm:3MM}
There are exactly three obstructions to knotless embedding of size 23.
\end{theorem}
Since there are no obstructions of size 20 or 
less, 14 of size 21, 92 of size 22 and at least 156 of size 28 (see \cite{FMMNN, GMN}),
the theorem shows that the knotless embedding obstruction set has a dip at small size, 23. 

The proof of Theorem~\ref{thm:3MM} naturally breaks into two parts.
We show ``by hand'' that the three obstruction graphs have no
knotless embedding and are minor minimal for that property.  
To show these three are the {\bf only} obstructions of size 23 we make use
of the connection with 2-apex graphs. 

\begin{figure}[htb]
\centering
\includegraphics[scale=1]{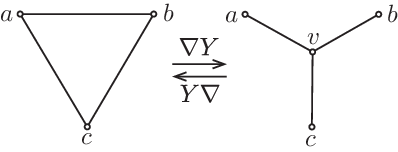}
\caption{$\ty$ and $\yt$ moves.} 
\label{fig:TY}
\end{figure}

This second part of the argument 
is novel. While our analysis depends on computer calculations,
the resulting observations may be of independent interest.
To describe these, we 
review some terminology for graph families, see~\cite{GMN}.
The {\em family} of graph $G$ is
the set of graphs related to $G$ by a sequence of zero or 
more $\ty$ and $\yt$
moves, see Figure~\ref{fig:TY}. The graphs
in $G$'s family are {\em cousins} of $G$. We do not allow $\yt$
moves that would result in doubled edges and all cousins
have the same size. If a $\ty$ move on $G$ results in graph $H$, we 
say $H$ is a {\em child} of $G$ and $G$ is a {\em parent} of $H$. 
The set of graphs that can be obtained from $G$ by a sequence of $\ty$
moves are $G$'s {\em descendants}. Similarly, the set of graphs
that can be obtained from $G$ by a sequence of $\yt$ moves
are $G$'s {\em ancestors}.

To show that the three graphs are the only obstructions of size 23 
relies on a careful
analysis of certain graph families
with respect to knotless embedding.
This analysis includes progress in resolving the following question.
\begin{question}
\label{que:d3MMIK}
If $G$ has a vertex of degree less than three, can an ancestor or descendant of $G$ be an obstruction
for knotless embedding?
\end{question}

It has been fruitful to search in graph families for obstructions to 
knotless embedding. For example, 
of the 264 known obstructions described in \cite{FMMNN},
all but three occur as part of four graph families. The same paper states
``It is natural to investigate the graphs obtained by adding one edge to
each of ... six graphs'' in the family of the Heawood graph. We carry
out this investigation and classify, with respect to knotless embedding, the graphs obtained 
by adding an edge to a Heawood family graph;
see Section 3 for details. 
As a first step toward a more general strategy for this type of problem,
we make the following connections between sets of graphs
of the form $G+e$ obtained by adding an edge to graph $G$.
\begin{theorem} 
\label{thm:Gpe}
If $G$ is a parent of $H$, then every $G+e$ has a cousin that is an $H+e$.
\end{theorem}

\begin{corollary}
\label{cor:Gpe}
If $G$ is an ancestor of $H$, then every $G+e$ has a cousin that is an $H+e$.
\end{corollary}
	
In Section 4, we prove the following.
\begin{theorem}
\label{thm:ord10}
There are exactly 35 obstructions to knotless embedding of order ten.
\end{theorem}
This depends on a classification of the maximal knotless graphs of order ten,
that is the graphs that are edge maximal in the set of graphs
that admit a knotless embedding, see~\cite{EFM}. In Appendix~\ref{sec:appmnik} we show
that there are 49 maximal knotless graphs of order ten.

In contrast to graph size, distributions with respect to graph order
generally do not have dips or gaps. In particular, Theorem~\ref{thm:ord10} continues an increasing trend of
no obstructions of order 6 or less,
one obstruction of order 7~\cite{CG}, two of order 8~\cite{CMOPRW,BBFFHL}, and eight of order 9~\cite{MMR}.

In the next section we discuss some graph properties for which we 
know something about the obstruction set, with a focus on those
that have a dip at small size. In Sections 3 and 4, we prove 
Theorems~\ref{thm:3MM} and \ref{thm:ord10}, respectively. 
Appendix A is a traditional
proof that the graphs $G_1$ and $G_2$ are IK. In Appendix B 
we describe the 49 maxnik graphs of order ten. Finally, Appendix 
C gives graph6~\cite{sage} notation for the important graphs
and further details of arguments throughout the paper including
the structure of the large families that occur at the end of subsection~\ref{sec:Heawood}.
	
\section{Dips at small size}
	
As mentioned in the introduction, it remains difficult to determine the
obstruction set even for simple graph properties. In this 
section we survey some topological graph properties for which we 
know something about the obstruction set. 

We begin by focusing on four properties that feature a prominent 
dip or gap at small sizes.
Two are the obstructions to knotless and torus embeddings mentioned in the 
introduction. Although the list of torus obstructions is likely incomplete
we can be confident about the dip at size 20. Like all of the incomplete 
sets we look at, research has focused on smaller sizes such that 
data on this side of the distribution is (nearly) complete. In the specific case
of torus obstructions, we can compare with a 2002 study~\cite{C} that listed 16,682
torus obstructions. Of the close to one thousand graphs added to the set 
in the intervening decade and a half, only three are of size 23 or smaller.

Similarly, while our proof that there are three obstructions for knotless
embedding depends on computer verification, it seems certain that the
number of obstructions at size 23 is far smaller than the 92 of size 22 and the 
number for size 28, which is known to be at least 156~\cite{GMN}. 

	\begin{table}
	    \centering
	    \begin{tabular}{l|ccccccc}
	        Size & 15 & 16 & 17 & 18 & 19 & 20 & 21 \\ \hline
	        Count & 7 & 0 & 0 & 4 & 5 & 22 & 33  
	    \end{tabular}
	    \caption{Count of apex obstructions through size 21.}
	    \label{tab:MMNASiz}
	\end{table}
	
The set of apex obstructions, investigated by  \cite{JK,LMMPRTW,P},
suggests one possible source for these dips. 
A graph is {\em apex} if it is planar, or becomes
planar on deletion of a single vertex.
As yet, we do not have a complete listing of the apex obstruction
set, but Jobson and K{\'e}zdy~\cite{JK} report that there
are at least 401 obstructions. Table~\ref{tab:MMNASiz} shows the 
classification of obstructions through size 21 obtained by Pierce~\cite{P} in his senior thesis. There is a noticeable gap at sizes 16 and 17.
The seven graphs of size 15 are the graphs in the Petersen family. This 
is the family of the Petersen graph, which is also the obstruction
set for linkless embedding. Note that, as for all our tables of size 
distributions, the table begins with what is known to be the smallest
size in the distribution, in this case 15.

Our proof that there are only three knotless embeddding obstructions
of size 23 depends on a related property, $2$-apex.
We say that a graph is {\em 2-apex} if it is apex, or becomes apex on deletion of a single vertex. Table~\ref{tab:MMN2ASiz} shows the little that we know about the obstruction set for this family \cite{MP}.
Aside from the counts for sizes 21 and 22 and the gap at size 23, we know only that there are obstructions for each size from 24 through 30.
	
	\begin{table}
	    \centering
	    \begin{tabular}{l|ccc}
	        Size & 21 & 22 & 23  \\ \hline
	        Count & 20 & 60 & 0  
	    \end{tabular}
	    \caption{Count of 2-apex obstructions through size 23.}
	    \label{tab:MMN2ASiz}
	\end{table}

While it is not an explanation, these four properties with notable dips or 
gaps are related to one another and seem to stem from the Petersen graph,
a notorious  counterexample in graph theory.
The most noticeable gap is 
for the apex property and the seven graphs to the left of the gap are precisely 
the Petersen family. The gap at size 23 for the $2$-apex property is doubtless
related. In turn, our proof of Theorem~\ref{thm:3MM} in Section 3 relies
heavily on the strong connection between $2$-apex and knotless graphs.
For this reason, it is not surprising that the gap at size 23 for
$2$-apex obstructions results in a similar dip at size 23 
for obstructions for knotless embeddings.

The connection with the 
dip at size 20 for torus embeddings is not as direct,  
but we remark that eight of the obstructions of size 19 have a
minor that is either a graph in the Petersen family, or else
one of those seven graphs with a single edge deleted.

In contrast, let us briefly mention some well known obstruction sets that
do not have dips or gaps and are instead unimodal. 
There are two obstructions to planarity, one each of size nine ($K_{3,3}$) and ten ($K_5$). 
The two obstructions, $K_4$ and $K_{3,2}$, to outerplanarity both have size six and the seven obstructions to linkless embedding in the 
Petersen family~\cite{RST} are all of size 15. 

Aside from planarity and linkless embedding,
the most famous set is likely the 35 obstructions to 
projective planar embedding~\cite{A,GHW,MT}.
Table~\ref{tab:PPSiz} shows that the size distribution for these 
obstructions is unimodal.
	
	\begin{table}
	    \centering
	    \begin{tabular}{l|cccccc}
	        Size & 15 & 16 & 17 & 18 & 19 & 20  \\ \hline
	        Count & 4 & 7 & 10 & 10 & 2 & 2   
	    \end{tabular}
	    \caption{Count of projective planar obstructions by size.}
	    \label{tab:PPSiz}
	\end{table}

    
	
\section{Knotless embedding obstructions of size 23}
    
In this section we prove Theorem~\ref{thm:3MM}: 
there are exactly three obstructions 
to knotless embedding of size 23. 
Along the way (see subsection~\ref{sec:Heawood}) we provide evidence in support of a negative answer to
Question~\ref{que:d3MMIK} and prove Theorem~\ref{thm:Gpe} and its
corollary. We also classify, with respect to knotless embedding, three
graph families of size 22. These families include every graph
obtained by adding an edge to a Heawood family graph.

We begin with some terminology.
A graph that admits no knotless embedding is {\em intrinsically knotted (IK)}. In contrast, we will call the graphs that admit a knotless 
embedding {\em not intrinsically knotted (nIK)}. If $G$ is
in the obstruction set for knotless embedding we will say 
$G$ is {\em minor minimal intrinsically knotted (MMIK)}. 
This reflects that, while $G$ is IK, no proper minor
of $G$ has that property. Similarly, we will call 2-apex obstructions {\em minor minimal not 2-apex (MMN2A)}.
    
Our strategy for classifying MMIK graphs of size 23 is based on
the following observation.

\begin{lemma} \cite{BBFFHL,OT} \label{lem:2apex}
If $G$ is 2-apex, then $G$ is not IK
\end{lemma}

Suppose $G$ is MMIK of size 23. By Lemma~\ref{lem:2apex}, $G$ is 
not 2-apex and, therefore, $G$ has an MMN2A minor. The MMN2A graphs
through order 23 were classified in~\cite{MP}. All but eight of them
are also MMIK and none are of size 23. It follows that a MMIK graph
of size 23 has one of the eight exceptional MMN2A graphs as a minor.
Our strategy is to construct all size 23 expansions of the eight
exceptional graphs and determine which of those is in fact MMIK.
    
Before further describing our search, we remark that it does rely 
on computer support. Indeed, the initial classification of MMN2A
graphs in \cite{MP} is itself based on a computer search. 
We give a traditional proof that there are three size 23 MMIK graphs,
which is stated as Theorem~\ref{thm:TheThree} below.
We rely on computers only for the argument that there are 
no other size 23 MMIK graphs. Note that, even if 
we cannot provide a complete, traditional proof 
that there are no more than three size 23 MMIK graphs, 
our argument does strongly suggest that there are 
far fewer MMIK graphs of size 23 than the
known 92 MMIK graphs of size 22 and at least 156 of size 28~\cite{FMMNN, GMN}. 
In other words, even without computers, we have compelling evidence that 
there is a dip at size 23 for the obstructions to knotless embedding.
    
Below we give graph6 notation~\cite{sage} and edge lists for the three MMIK graphs of size 23. See also Figures~\ref{fig:G1} and \ref{fig:G2} in Appendix~\ref{sec:appG12}.
    
    \noindent%
    $G_1$ \verb"J@yaig[gv@?"
    $$[(0, 4), (0, 5), (0, 9), (0, 10), (1, 4), (1, 6), (1, 7), (1, 10), (2, 3), (2, 4), (2, 5), (2, 9),$$
    $$ (2, 10), (3, 6), (3, 7), (3, 8), (4, 8), (5, 6), (5, 7), (5, 8), (6, 9), (7, 9), (8, 10)]$$
    
    \noindent%
    $G_2$ \verb"JObFF`wN?{?"
    $$[(0, 2), (0, 4), (0, 5), (0, 6), (0, 7), (1, 5), (1, 6), (1, 7), (1, 8), (2, 6), (2, 7), (2, 8),$$
    $$(2, 9), (3, 7), (3, 8), (3, 9), (3, 10), (4, 8), (4, 9), (4, 10), (5, 9), (5, 10), (6, 10)]$$
    
    \noindent%
    $G_3$ \verb"K?bAF`wN?{SO"
    $$[(0, 4), (0, 5), (0, 7), (0, 11), (1, 5), (1, 6), (1, 7), (1, 8), (2, 7), (2, 8), (2, 9), (2, 11),$$
    $$(3, 7), (3, 8), (3, 9), (3, 10), (4, 8), (4, 9), (4, 10), (5, 9), (5, 10), (6, 10), (6, 11)]$$
    
The graph $G_1$ was discovered by Hannah Schwartz~\cite{N}. Graphs
$G_1$ and $G_2$ have order 11 while $G_3$ has order 12. We prove
the following in subsection~\ref{sec:3MMIKpf} below.
    
\begin{theorem} \label{thm:TheThree}
The graphs $G_1$, $G_2$, and $G_3$ are MMIK of size 23
\end{theorem}

We next describe the computer search that shows there are no other size 
23 MMIK graphs and completes the proof of 
Theorem~\ref{thm:3MM}. We also describe how Question~\ref{que:d3MMIK}, 
Theorem~\ref{thm:Gpe} and its corollary fit in, along with the 
various size 22 families.

There are eight exceptional graphs of size at most 23
that are MMN2A and not MMIK. Six of them are in the Heawood family of size 21 graphs. 
The other two, $H_1$ and $H_2$, are 4-regular graphs on 11 vertices with size 22 described in \cite{MP}, listed
in Appendix C, and shown in Figure~\ref{fig:H1H2}.

It turns out that the three graphs of Theorem~\ref{thm:TheThree} are
expansions of $H_1$ and $H_2$. 
In subsection~\ref{sec:1122graphs} we show that these two graphs are 
MMN2A but not IK and argue that no other size 23 expansion 
of $H_1$ or $H_2$ is MMIK.
    
The Heawood family consists of twenty graphs of size 21
related to one another by $\ty$ and $\yt$ moves, see Figure~\ref{fig:TY}.
In \cite{GMN,HNTY} two groups, working independently,
verified that 14 of the graphs in the family are MMIK, and the remaining six are MMN2A and not MMIK. 
In subsection~\ref{sec:Heawood}
below, we argue that no size 23 expansion of any of these six Heawood family
graphs is MMIK.
Combining the arguments of the next three subsections give a proof of 
Theorem~\ref{thm:3MM}.


There are two Mathematica programs written by 
Naimi and available at his website~\cite{NW} that we use throughout. 
One, isID4, is an implementation of the algorithm of 
Miller and Naimi~\cite{MN} and we refer the reader to that paper for details. Note that, while this
algorithm can show that a particular graph is IK, it does not allow us to deduce that a graph is nIK.

Instead, we make use of a second program of Naimi, findEasyKnots, to find knotless embeddings of nIK 
graphs.
This program determines the set of cycles $\Sigma$ of a graph $G$. Then, given
an embedding of $G$, for each $\sigma \in \Sigma$ the program applies $R1$ and $R2$ Reidemeister moves 
(see~\cite{R}) until it arrives at one of three possible outcomes: $\sigma$ is the unknot; $\sigma$
is an alternating (hence non-trivial) knot; or
$\sigma$ is a non-alternating knot (which may or may not be trivial). In this paper, we will often show
that a graph is nIK by presenting a knotless embedding. In all cases, this means that when we 
apply findEasyKnots to the embedding, it determines that every cycle in the graph is a trivial knot.

Before diving into our proof of Theorem~\ref{thm:3MM}, we 
state a few lemmas we will use throughout. The first is about the 
{\em minimal degree} $\delta(G)$, which is the least degree
among the vertices of graph $G$.

\begin{lemma} If $G$ is MMIK, then $\delta(G) \geq 3$.
\label{lem:delta3}
\end{lemma}

\begin{proof} Suppose $G$ is IK with $\delta(G) < 3$.
By either deleting, or contracting an edge on, a 
vertex of small degree, we find a proper minor that
is also IK. 
\end{proof}

\begin{figure}[htb]
\centering
\includegraphics[scale=1]{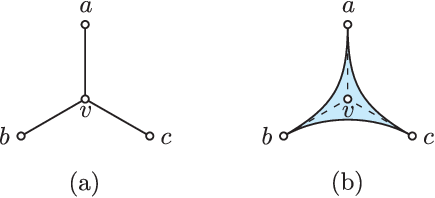}
\caption{Place a triangle in a neighborhood of
the $Y$ subgraph.} 
\label{fig:TYn}
\end{figure}

\begin{lemma}
\label{lem:tyyt}
The $\ty$ move preserves IK: If $G$ is IK and $H$ is obtained from $G$ by a $\ty$ move, then $H$ is also IK. 
Equivalently, the $\yt$ move preserves
nIK: if $H$ is nIK and $G$ is obtained from $H$ by a $\yt$ move, then $G$ is also nIK.
\end{lemma}

\begin{proof} 
We begin by noting that the $\yt$ move preserves
planarity. Suppose $H$ is planar and has an induced
$Y$ or $K_{3,1}$ subgraph with degree three vertex $v$ adjacent to vertices $a,b,c$. In a planar
embedding of $H$ we can choose a neighborhood of the $Y$ subgraph small
enough that it excludes all other edges of $H$, see Figure~\ref{fig:TYn}. 
This allows us to place a $3$-cycle $a,b,c$ within the
neighborhood which shows that the graph $G$ that results from a $\yt$ move is also planar.

Now, suppose $H$ is nIK with degree three vertex $v$.
Then, as in Figure~\ref{fig:TYn}, in a knotless embedding of $H$, we can find a neighborhood of the
induced $Y$ subgraph small enough that it intersects no other edges of $H$. Again, place a $3$-cycle 
$a,b,c$ within this neighborhood. We claim that the resulting embedding of $G$ obtained by this $\yt$
move is likewise knotless. Indeed, any cycle in $G$ that uses vertices $a$, $b$, or $c$
has a corresponding cycle in $H$ that differs only by a small move within the neighborhood of the $Y$
subgraph. 
Since every cycle of $H$ is unknotted, the same is true for every cycle in this embedding of $G$.
\end{proof}

Finally, we note that the MMIK property can move backwards
along $\ty$ moves.

\begin{lemma} \cite{BDLST,OT}
\label{lem:MMIK}
Suppose $G$ is IK and 
$H$ is obtained from $G$ by a $\ty$ move.
If $H$ is MMIK, then $G$ is also MMIK.
\end{lemma}

\subsection{Proof of Theorem~\ref{thm:TheThree}}\label{sec:3MMIKpf}\ 

In this subsection we prove Theorem~\ref{thm:TheThree}: the three graphs $G_1$, $G_2$, and $G_3$ are MMIK. 

We first show these graphs are IK. 
For $G_1$ and $G_2$ we present a proof ``by hand"
as Appendix~\ref{sec:appG12}. We remark that we can also verify that these two graphs are IK
using Naimi's Mathematica implementation~\cite{NW} of 
the algorithm of Miller and Naimi~\cite{MN}.

The graph $G_3$ is obtained from $G_2$ by a single $\ty$ move.
Specifically, using the edge list for $G_2$ given above:
$$[(0, 2), (0, 4), (0, 5), (0, 6), (0, 7), (1, 5), (1, 6), (1, 7), (1, 8), (2, 6), (2, 7), (2, 8),$$
$$(2, 9), (3, 7), (3, 8), (3, 9), (3, 10), (4, 8), (4, 9), (4, 10), (5, 9), (5, 10), (6, 10)],$$
make the $\ty$ move on the triangle $(0,2,6)$.
Since $G_2$ is IK, Lemma~\ref{lem:tyyt} implies $G_3$ is also IK.

To complete the proof of Theorem~\ref{thm:TheThree}, it remains only to show that all proper minors of $G_1$, $G_2$ and $G_3$ are nIK.

First we argue that no proper minor of $G_1$ is IK.
Up to isomorphism, there are 12 minors obtained by contracting or deleting a single edge. 
Each of these is 2-apex, except for the MMN2A graph $H_1$.
By Lemma~\ref{lem:2apex}, a 2-apex graph is nIK and Figure~\ref{fig:H1H2} gives a knotless embedding of $H_1$
(as we have verified using Naimi's program findEasyKnots, see~\cite{NW}). 
This shows that all proper minors of $G_1$ are nIK.

\begin{figure}[htb]
\centering
\includegraphics[scale=1]{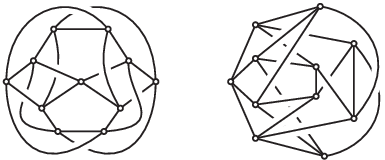}
\caption{Knotless embeddings of graphs $H_1$ (left) and $H_2$ (right).} 
\label{fig:H1H2}
\end{figure}

Next we argue that no proper minor of $G_2$ is IK.
Up to isomorphism, there are 26 minors obtained by deleting or contracting an edge of $G_2$. 
Each of these is 2-apex, except for the MMN2A graph $H_2$.
Since $H_2$ has a knotless embedding as shown in Figure~\ref{fig:H1H2}, 
then, similar to the argument for $G_1$,
all proper minors of $G_2$ are nIK.

It remains to argue that no proper minor of $G_3$ is IK.
Up to isomorphism, there are 26 minors obtained by deleting or contracting an edge of $G_3$. 
Each of these is 2-apex, except for the MMN2A graph $H_2$ which is nIK.
Similar to the previous cases, all proper minors of $G_3$ are nIK.
This completes the proof of Theorem~\ref{thm:TheThree}.

\subsection{Expansions of nIK Heawood family graphs}\label{sec:Heawood}\ 

In this subsection we argue that there are no MMIK graphs among the
size 23 expansions of the six nIK graphs in the Heawood family. 
As part of the argument, we classify with respect to knotless embedding
all graphs obtained by adding an edge to a Heawood family graph. We also
discuss our progress on Question~\ref{que:d3MMIK} and prove Theorem~\ref{thm:Gpe}
and its corollary.

We will use the notation of \cite{HNTY} 
to describe the twenty graphs in the Heawood family,
which we also recall in Appendix C.
For the reader's convenience, Figure~\ref{fig:Hea}
shows four of the graphs in the family
that are central to our discussion.

\begin{figure}[htb]
\centering
\includegraphics[scale=1]{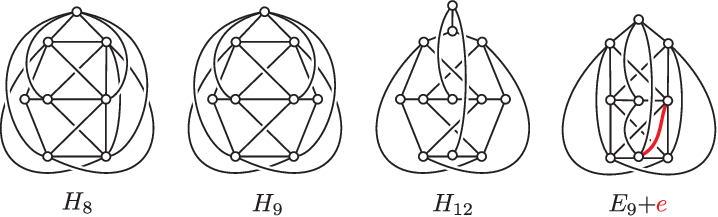}
\caption{Four graphs in the Heawood family.
($E_9$ is in the family, but $E_9+e$ is not.)} 
\label{fig:Hea}
\end{figure}

Kohara and Suzuki~\cite{KS} showed that 14 
graphs in this family are MMIK. The remaining six,
$N_9$, $N_{10}$, $N_{11}$, $N'_{10}$, $N'_{11}$, and $N'_{12}$, are nIK~\cite{GMN,HNTY}. The 
graph $N_9$ is called $E_9$ in \cite{GMN}. 
In this subsection
we argue that no size 23 expansion of these six graphs is MMIK.


The Heawood family graphs are the cousins of the Heawood graph,
which is denoted $C_{14}$ in \cite{HNTY}.
All have size 21. We can expand a graph to one
of larger size either by adding an edge or by splitting a vertex.
In {\em splitting a vertex} we replace a graph $G$ with
a graph $G'$ so that the order increases by one: $|G'| = |G|+1$.
This means we replace a vertex $v$ of $G$ with two vertices 
$v_1$ and $v_2$ in $G'$ and identify the 
remaining vertices of $G'$  with those of 
$V(G) \setminus \{v\}$.
As for edges, $E(G')$ includes the 
edge $v_1v_2$. In addition, we require that
the union of the neigborhoods of $v_1$ and $v_2$
in $G'$ otherwise agrees with the neighborhood of $v$:
$N(v) = N(v_1) \cup N(v_2) \setminus \{v_1,v_2 \}$. In other words,
$G$ is the result of contracting $v_1v_2$ in $G'$ where double 
edges are suppressed: $G = G'/v_1v_2$.

Our goal
is to argue that there is no size 23 MMIK graph that
is an expansion of one of the six nIK Heawood family graphs, $N_9$, $N_{10}$, $N_{11}$, $N'_{10}$, $N'_{11}$, and $N'_{12}$.
As a first step, we will argue that, if there were 
such a size 23 MMIK expansion, it would also be an 
expansion of one of 29 nIK graphs of size 22.

Given a graph $G$, we will use $G+e$ to denote a graph obtained by
adding an edge $e \not\in E(G)$. 
As we will show, if $G$ is a Heawood family graph, then $G+e$ will fall
in one of three families that we will call the $H_8+e$ family, 
the $E_9+e$ family, and the $H_9+e$ family. The $E_9+e$ family is
discussed in \cite{GMN} where it is shown to consist of 110 graphs, all IK. 

The $H_8 + e$ graph is formed by adding an edge to the 
Heawood family graph $H_8$ between two of its degree 5 vertices. 
The $H_8+e$ family consists of 125 graphs, 29 of which are nIK and the remaining 96 are IK, as we will now
argue. For this, we leverage 
graphs in the Heawood family. In addition to $H_8+e$, 
the family includes an $F_9+e$ graph formed by 
adding an edge between the two degree 3 vertices of $F_9$. 
Since $H_8$ and $F_9$ are both IK~\cite{KS}, 
the corresponding
graphs with an edge added are as well. By 
Lemma~\ref{lem:tyyt}, $H_8+e$, $F_9+e$ and all their
descendants are IK. These are the 96 IK graphs in the family.

The remaining 29 graphs are all ancestors of six 
graphs that we describe below. Once we establish that
these six are nIK, then Lemma~\ref{lem:tyyt} ensures
that all 29 are nIK. 

We will denote the six graphs
$T_i$, $i = 1,\ldots,6$ where we have used the letter $T$ 
since five of them have a degree two vertex
(`T' being the first letter of `two').
After contracting an edge on the degree 2 vertex, we
recover one of the nIK Heawood family graphs, $N_{11}$ or $N'_{12}$. It follows that these five graphs 
are also nIK.

The two graphs that become $N_{11}$ after 
contracting an edge have the following graph6 notation\cite{sage}:

$T_1$: \verb'KSrb`OTO?a`S' $T_2$: \verb'KOtA`_LWCMSS'

The three graphs that contract to $N'_{12}$ are:

$T_3$: \verb'LSb`@OLOASASCS' $T_4$: \verb'LSrbP?CO?dAIAW' $T_5$: \verb'L?tBP_SODGOS_T'

The five graphs we have described so far along with
their ancestors account for 26 of the nIK graphs in the
$H_8+e$ family. The remaining three are ancestors of

$T_6$: \verb'KSb``OMSQSAK'

Figure~\ref{fig:T6} shows a knotless embedding of $T_6$.
By Lemma~\ref{lem:tyyt}, its two ancestors 
are also nIK and this completes the 
count of 29 nIK graphs in the $H_8+e$ family.

\begin{figure}[htb]
\centering
\includegraphics[scale=1]{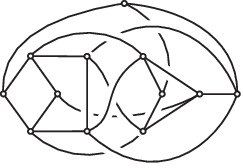}
\caption{A knotless embedding of the $T_6$ graph.} 
\label{fig:T6}
\end{figure}

The graph $H_9+e$ is formed by adding an edge to $H_9$
between the two degree 3 vertices. There are five graphs
in the $H_{9}+e$ family, four of which are $H_9+e$ and its 
descendants. Since $H_9$ is IK~\cite{KS}, by 
Lemma~\ref{lem:tyyt}, these four graphs
are all IK. The remaining graph in the family is
the MMIK graph denoted $G_{S}$ in \cite{FMMNN} and shown in Figure~\ref{fig:HS}. 
Although the graph is credited to Schwartz in that 
paper, it was a joint discovery of Schwartz and 
and Barylskiy~\cite{N}. 
Thus, all five graphs in the $H_9+e$ family are IK.

\begin{figure}[htb]
\centering
\includegraphics[scale=1]{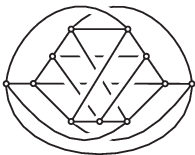}
\caption{The MMIK graph $G_S$.} 
\label{fig:HS}
\end{figure}

Having classified the graphs in the three families with respect 
to intrinsic knotting, using Corollary~\ref{cor:3fam} below,
we have completed the investigation, suggested by \cite{FMMNN}, of graphs formed
by adding an edge to a Heawood family graph.

We remark that among the three families $H_8+e$, $E_9+e$, and $H_9+e$, 
the only instances of a graph with a degree 2 vertex occur in the family $H_8+e$, which also contains no MMIK graphs. This observation suggests the following question.

\setcounter{section}{1}
\setcounter{theorem}{1}

\begin{question}
If $G$ has minimal degree $\delta(G) < 3$,
is it true that $G$'s ancestors and descendants include no MMIK graphs?
\end{question} 

\setcounter{section}{3}
\setcounter{theorem}{5}

Initially, we suspected that such a $G$ has no MMIK cousins at all. 
However, we discovered that the MMIK graph of size 26, described in 
Section~\ref{sec:ord10} below, includes graphs of minimal degree two
among its cousins. 
Although we have not completely resolved the question, we have two partial results.

\begin{theorem} If $\delta(G) < 3$ and $H$ is a descendant of $G$, then $H$ is not MMIK.
\end{theorem}

\begin{proof} Since $\delta(G)$ is non-increasing under
the $\ty$ move, $\delta(H) \leq \delta(G) < 3$ and
$H$ is not MMIK by Lemma~\ref{lem:delta3}.
\end{proof}

As defined in \cite{GMN} a graph has a $\bar{Y}$ if
there is a degree 3 vertex that is also part of a $3$-cycle.
A $\yt$ move at such a vertex would result in doubled 
edges.

\begin{lemma}\label{lem:ybar} A graph with a $\bar{Y}$ is not MMIK.
\end{lemma}

\begin{proof} Let $G$ have a vertex $v$ with 
$N(v) = \{a,b,c\}$ and $ab \in E(G)$. We can assume
$G$ is IK. Make a $\ty$ move on triangle $v,a,b$ to 
obtain the graph $H$. By Lemma~\ref{lem:tyyt}
$H$ is IK, as is the homeomorphic graph $H'$ obtained
by contracting an edge at the degree 2 vertex $c$.
But $H' = G - ab$ is obtained by deleting an 
edge $ab$ from $G$. Since $G$ has a proper subgraph
$H'$ that is IK, $G$ is not MMIK.
\end{proof}






\begin{theorem} If $G$ has a child $H$ with $\delta(H) < 3$,
then $G$ is not MMIK.
\end{theorem}

\begin{proof} By Lemma~\ref{lem:delta3}, we can assume
$G$ is IK with $\delta(G) = 3$. It follows that $G$ has
a $\bar{Y}$ and is not MMIK by the previous lemma.
\end{proof}


Suppose $G$ is a size 23 MMIK expansion of one 
of the six nIK Heawood family graphs. 
We will argue that $G$ must be an expansion
of one of the graphs in the three families, $H_8+e$,
$E_9+e$, and $H_9+e$. However, as a MMIK graph,
$G$ can have no size 22 IK minor. Therefore, $G$ must
be an expansion of one of the 29 nIK graphs in the 
$H_8+e$ family.

There are two ways to form a size 22 expansion of one of 
the six nIK graphs, either add an edge or split a vertex.
We now show that if $H$ is in the Heawood family, then 
$H+e$ is in one of the three families,
$H_8+e$, $E_9+e$, and $H_9 + e$. 
We begin with a proof of a theorem and corollary,
mentioned in the introduction, that describe how adding an edge to 
a graph $G$ interacts with the graph's family.

\setcounter{section}{1}
\setcounter{theorem}{2}

\begin{theorem} 
If $G$ is a parent of $H$, then every $G+e$ has a cousin that is an $H+e$.
\end{theorem}

\begin{proof}
Let $H$ be obtained by a $\ty$ move that replaces the triangle $abc$ in $G$ with three edges on the new vertex $v$. That is,
$V(H) = V(G) \cup \{v\}$.
Form $G+e$ by adding the edge $e = xy$. Since $V(H) = V(G) \cup \{v\}$, then $x,y \in V(H)$ and the graph $H+e$ is a 
cousin of $G+e$ by a $\ty$ move on the triangle $abc$.
\end{proof}

\begin{corollary}
If $G$ is an ancestor of $H$, then every $G+e$ has a cousin that is an $H+e$.
\end{corollary}

\setcounter{section}{3}
\setcounter{theorem}{8}


Every graph in the Heawood family is an ancestor of one of two graphs, the Heawood graph (called $C_{14}$ in \cite{HNTY}) and the graph $H_{12}$ (see Figure~\ref{fig:Hea}).



\begin{theorem}
\label{thm:Heawpe}
Let $H$ be the Heawood graph. Up to isomorphism, there are two $H+e$ graphs. One
is in the $H_8+e$ family, the other in the $E_9+e$ family.
\end{theorem}

\begin{proof}
The diameter of the Heawood graph is three. Up to isomorphism, we can either add an edge between vertices of distance two or three.
If we add an edge between vertices of distance two, 
the result is a graph in the $H_8+e$ family. If the distance
is three, we are adding an edge between the different parts and the result is a bipartite graph of size 22. 
As shown in~\cite{KMO}, this means it is cousin 89 of the
$E_9+e$ family.
\end{proof}

\begin{theorem} 
\label{thm:H12pe}
Let $G$ be formed by adding an edge to $H_{12}$. 
Then $G$ is in the $H_8+e$, $E_9+e$, or $H_9+e$ family.
\end{theorem}
 
\begin{proof}
Note that $H_{12}$ consists of six degree 4 vertices and six degree 3 vertices.
Moreover, five of the degree 3 vertices are created by $\ty$ moves in the process of obtaining $H_{12}$ from $K_7$.
Let $a_i$ ($i = 1 \ldots 5$) denote those five degree 3 vertices.
Further assume that $b_1$ is the remaining degree 3 vertex and $b_2$, $b_3$, $b_4$, $b_5$, $b_6$ and $b_7$ are 
the remaining degree 4 vertices.
Then the $b_j$ vertices correspond to vertices of $K_7$ before applying the  $\ty$ moves.

First suppose that $G$ is obtained from $H_{12}$ by adding an edge which connects two $b_j$ vertices.
Since these seven vertices are the vertices of $K_7$ before using $\ty$ moves, there is exactly one vertex among the $a_i$, say $a_1$, that is adjacent to the two endpoints of the added edge.
Let $G'$ be the graph obtained from $G$ by applying 
$\yt$ moves at $a_2$, $a_3$, $a_4$ and $a_5$.
Then $G'$ is isomorphic to $H_8+e$.
Therefore $G$ is in the $H_8+e$ family.

Next suppose that $G$ is obtained from $H_{12}$ by adding an edge which connects two $a_i$ vertices.
Let $a_1$ and $a_2$ be the endpoints of the added edge.
We assume that $G'$ is obtained from $G$ by using 
$\yt$ moves at $a_3$, $a_4$ and $a_5$.
Then there are two cases: either $G'$ is obtained from $H_9$ or $F_9$ by adding an edge which connects two degree 3 vertices.
In the first case, $G'$ is isomorphic to $H_9+e$.
Thus $G$ is in the $H_9+e$ family.
In the second case, $G'$ is in the $H_8+e$ or $E_9+e$ family by Corollary~\ref{cor:Gpe} and Theorem~\ref{thm:Heawpe}.
Thus $G$ is in the $H_8+e$ or $E_9+e$ family.

Finally suppose that $G$ is obtained from $H_{12}$ by adding an edge which connects an $a_i$ vertex and 
a $b_j$ vertex.
Let $a_1$ be a vertex of the added edge.
We assume that $G'$ is the graph obtained from $G$ by using $\yt$ moves at $a_2$, $a_3$, $a_4$ and $a_5$.
Since $G'$ is obtained from $H_8$ by adding an edge, $G'$ is in the $H_8+e$ or $E_9+e$ family by Corollary~\ref{cor:Gpe} and Theorem~\ref{thm:Heawpe}.
Therefore $G$ is in the $H_8+e$ or $E_9+e$ family.
\end{proof} 

\begin{corollary} 
\label{cor:3fam}
If $H$ is in the Heawood family, then $H+e$ is in the $H_8+e$, $E_9+e$, or $H_9+e$ family.
\end{corollary}

\begin{proof} The graph $H$ is either an ancestor of the Heawood graph or $H_{12}$. Apply Corollary~\ref{cor:Gpe} and
Theorems~\ref{thm:Heawpe} and \ref{thm:H12pe}.
\end{proof}

\begin{corollary} 
\label{cor:29nIK}
If $H$ is in the Heawood family and $H+e$ is nIK, then $H+e$ is one of the 29 nIK graphs in the 
$H_8+e$ family
\end{corollary}

\begin{lemma}
\label{lem:deg2}
Let $H$ be a nIK Heawood family graph and $G$ be an expansion obtained by splitting a vertex of $H$. 
Then either $G$ has a vertex of degree at most two, or else it is in the $H_8+e$, $E_9+e$, or $H_{9}+e$ family.
\end{lemma}

\begin{proof}
Note that $\Delta(H) \leq 5$. If $G$ has no vertex 
of degree at most two, then the vertex split produces a vertex of degree three. 
A $\yt$ move on the degree three vertex produces $G'$ which is of the form $H+e$.
\end{proof}

\begin{corollary}
\label{cor:deg2}
Suppose $G$ is nIK and a size 22 expansion of a nIK Heawood family graph.
Then either $G$ has a vertex of degree at most two or $G$ is in the $H_8 + e$ family. 
\end{corollary}

\begin{theorem}
\label{thm:23to29}
Let $G$ be size 23 MMIK with a minor that is
a nIK Heawood family graph. Then $G$ is an expansion of 
one of the 29 nIK graphs in the $H_8+e$ family.
\end{theorem}

\begin{proof}
There must be a size 22 graph $G'$ intermediate to
$G$ and the Heawood family graph $H$. That is,
$G$ is an expansion of $G'$, which is an expansion
of $H$. By Corollary~\ref{cor:deg2}, we can
assume $G'$ has a vertex $v$ of degree at most two.

By Lemma~\ref{lem:delta3},
a MMIK graph has minimal degree, $\delta(G) \geq 3$.
Since
$G'$ expands to $G$ by adding an edge or splitting
a vertex, we conclude $v$ has degree two exactly
and $G$ is $G'$ with an edge added at $v$. 
Since $\delta(H) \geq 3$, this means $G'$ is obtained
from $H$ by a vertex split. 

In $G'$, let $N(v) = \{a,b\}$ and let $cv$ be
the edge added to form $G$. Then $H = G'/av$ and
we recognize $H+ac$ as a minor of $G$.
We are assuming $G$ is 
MMIK, so $H+ac$ is nIK and, by Corollary~\ref{cor:29nIK},
one of the 29 nIK graphs in the $H_8+e$ family. Thus,
$G$ is an expansion of $H+ac$, which is one of these 29
graphs, as required. 
\end{proof}

It remains to study the expansions of the 29 nIK graphs
in the $H_8+e$ family. We will give an overview of the 
argument, leaving many of the details to Appendix C.

The size 23 expansions of the 29 size 22 nIK graphs fall
into one of eight families, which we identify by the
number of graphs in the family: $\F_9$, $\F_{55}$
$\F_{174}$, $\F_{183}$, $\F_{547}$, $\F_{668}$,
$\F_{1229}$, and $\F_{1293}$. We list the graphs
in each family in Appendix C.

\begin{theorem} 
\label{thm:Fam8}
If $G$ is a size 23 MMIK expansion of a nIK
Heawood family graph, then $G$ is in one of the
eight families, $\F_9$, $\F_{55}$
$\F_{174}$, $\F_{183}$, $\F_{547}$, $\F_{668}$,
$\F_{1229}$, and $\F_{1293}$.
\end{theorem}

\begin{proof}
By Theorem~\ref{thm:23to29}, $G$ is an expansion 
of $G'$, which is one of the 29 nIK graphs 
in the $H_8+e$ family. As we have seen,
these 29 graphs are ancestors of the six graphs
$T_1, \ldots, T_6$. By Corollary~\ref{cor:Gpe},
we can find the $G'+e$ graphs by looking at the six 
graphs. Given the family listings in Appendix C, 
it is straightforward to verify that each $T_i+e$ is in
one of the eight families. This accounts for
the graphs $G$ obtained by adding an edge to 
one of the 29 nIK graphs in the $H_8+e$ family.

If instead $G$ is obtained by splitting a vertex
of $G'$, we use the strategy of Lemma~\ref{lem:deg2}.
By Lemma~\ref{lem:delta3}, $\delta(G) \geq 3$.
Since $\Delta(G') \leq 5$, 
the vertex split must produce a vertex of degree three. 
Then, a $\yt$ move on the degree three vertex produces
$G''$ which is of the form $G'+e$. Thus $G$ is a cousin 
of $G'+e$ and must be in one of the eight families.
\end{proof}

To complete our argument that there is no size 23 MMIK
graph with a nIK Heawood family minor, we
argue that there are no MMIK graphs in the 
eight families $\F_9, \F_{55}, \ldots, \F_{1293}$.
In large part our argument is based on two criteria
that immediately show a graph $G$ is not MMIK.
\begin{enumerate}
    \item $\delta(G) < 3$, see Lemma~\ref{lem:delta3}. 
    \item By deleting an edge, there is a proper minor $G-e$ that is an IK graph in the $H_8+e$, $E_9+e$,
    or $H_9+e$ families. In this case $G$ is IK, but not MMIK.
\end{enumerate}
By Lemma~\ref{lem:MMIK}, if $G$ has an ancestor
that satisfies criterion 2, then $G$ is also not MMIK.
By Lemma~\ref{lem:tyyt}, if $G$ has a nIK descendant, then
$G$ is also not nIK. 

\begin{theorem}There is no MMIK graph in the $\F_9$ family.
\end{theorem}

\begin{proof}
Four of the nine graphs satisfy the first criterion, $\delta(G) = 2$, and these are 
not MMIK by Lemma~\ref{lem:delta3}. The remaining 
graphs are descendants of a graph $G$ 
that is IK but not MMIK.
Indeed, $G$ satisfies criterion 2: by deleting an edge, we recognize $G-e$ as
an IK graph in the $H_9+e$ family (see Appendix C
for details). By 
Lemma~\ref{lem:MMIK}, $G$ and its descendants are also
not MMIK. 
\end{proof}

\begin{theorem}There is no MMIK graph in the $\F_{55}$ family.
\end{theorem}

\begin{proof}
All graphs in this family have $\delta(G) \geq 3$, 
so none satisfy the first criterion. 
All but two of the graphs in this family are not MMIK
by the second criterion. The remaining two graphs 
have a common parent that is IK but not MMIK. By 
Lemma~\ref{lem:MMIK}, these last two graphs are also 
not MMIK. See Appendix C for details.
\end{proof}

We remark that $\F_{55}$ is the only one of the eight 
families that has no graph with $\delta(G) < 3$.

\begin{theorem}There is no MMIK graph in the $\F_{174}$ family.
\end{theorem}

\begin{proof}
All but 51 graphs are not MMIK by the first criterion. Of 
the remaining graphs, all but 17 are not MMIK by the second
criterion. Of these, 11 are descendants of a graph 
$G$ that is IK but not MMIK by the second criterion. 
By Lemma~\ref{lem:MMIK}, these 11 are also not MMIK.
This leaves six graphs. For these we find two nIK
descendants. By Lemma~\ref{lem:tyyt} the remaining
six are also not MMIK. Both descendants have a degree
2 vertex. On contracting an edge of the degree 2 vertex,
we obtain a homeomorphic graph that is one of the 
29 nIK graphs in the $H_8+e$ family.
\end{proof}

\begin{theorem}There is no MMIK graph in the $\F_{183}$ family.
\end{theorem}

\begin{proof}
All graphs in this family have a vertex of degree two or 
less and are not MMIK by the first criterion.
\end{proof}

\begin{theorem}There is no MMIK graph in the $\F_{547}$ family.
\end{theorem}

\begin{proof}
All but 229 of the graphs in the family are not MMIK
by criterion one. Of those, all but 52 are not MMIK by
criterion two. Of those, 25 are ancestors of one 
of the graphs meeting criterion two and are not MMIK
by Lemma~\ref{lem:MMIK}. For the remaining 27 graphs,
all but five have a nIK descendant and are not IK
by Lemma~\ref{lem:tyyt}. For the remaining five,
three are ancestors of one of the five. In 
Figure~\ref{fig:547UK} we give knotless embeddings
of the other two graphs. Using Lemma~\ref{lem:tyyt},
all five graphs are nIK, hence not MMIK.
\end{proof}

\begin{figure}[htb]
\centering
\includegraphics[scale=1]{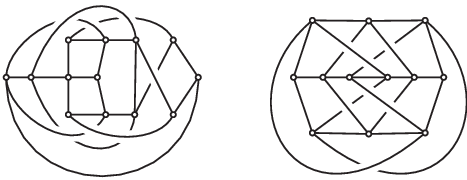}
\caption{Knotless embeddings of two graphs in $\F_{547}$.} 
\label{fig:547UK}
\end{figure}

\begin{theorem}There is no MMIK graph in the $\F_{668}$ family.
\end{theorem}

\begin{proof}
All but 283 of the graphs in the family are not MMIK
by criterion one. Of those, all but 56 are not MMIK by
criterion two. Of those, 23 are ancestors of one 
of the graphs meeting criterion two and are not MMIK
by Lemma~\ref{lem:MMIK}. For the remaining 33 graphs
all but three have a nIK descendant and are not IK
by Lemma~\ref{lem:tyyt}. Of the remaining three,
two are ancestors of the third. Figure~\ref{fig:668UK} 
is a knotless embedding of the common descendant. By 
Lemma~\ref{lem:tyyt} all three of these graphs are nIK, hence not MMIK.
\end{proof}

\begin{figure}[htb]
\centering
\includegraphics[scale=1]{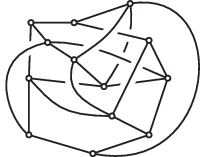}
\caption{Knotless embedding of a graph in $\F_{668}$.} 
\label{fig:668UK}
\end{figure}

\begin{theorem}There is no MMIK graph in the $\F_{1229}$ family.
\end{theorem}

\begin{proof}
There are 268 graphs in the family that are not MMIK
by criterion one. Of the remaining 961 graphs, all but 
140 are not MMIK by criterion two. Of those, all but three are
ancestors of one of the graphs meeting criterion two and are not MMIK
by Lemma~\ref{lem:MMIK}. The remaining three graphs have an IK minor
by contracting an edge and are, therefore, not MMIK.
\end{proof}

\begin{theorem}There is no MMIK graph in the $\F_{1293}$ family.
\end{theorem}

\begin{proof}
There are 570 graphs in the family that are not MMIK
by criterion one. Of the remaining 723 graphs, all but 
99 are not MMIK by criterion two. Of those, all but 12 are
ancestors of one of the graphs meeting criterion two and are not MMIK
by Lemma~\ref{lem:MMIK}. The remaining 12 graphs have an IK minor
by contracting an edge and are, therefore, not MMIK.
\end{proof}

\subsection{Expansions of the size 22 graphs $H_1$ and $H_2$}
\label{sec:1122graphs}
We have argued that a size 23 MMIK graph must have a minor that
is either one of six nIK graphs in the Heawood family, or else
one of two $(11,22)$ graphs that we call 
$H_1$: \verb'J?B@xzoyEo?' and $H_2$: \verb'J?bFF`wN?{?' 
(see Figure~\ref{fig:H1H2}).
We treated expansions of the Heawood family
graphs in the previous subsection. In this subsection we show
that $G_1$, $G_2$, and $G_3$ are the only size 23 MMIK expansions
of $H_1$ and $H_2$. Recall that these two graphs were shown 
MMN2A (minor minimal for $2$-apex) in \cite{MP}. The unknotted
embeddings of Figure~\ref{fig:H1H2} demonstrate that they are nIK.

By Lemma~\ref{lem:delta3}, if a vertex split of $H_1$ results in a vertex
of degree less than three, the resulting graph is not MMIK. Since
$H_1$ is $4$-regular, the only other way to make a vertex split
produces adjacent degree 3 vertices. Then, a $\yt$ move on one of the degree 
three vertices yields an $H_1+e$. Thus, a size 
23 MMIK expansion of $H_1$ must be in the family of a $H_1+e$.

Up to isomorphism, there are six $H_1+e$ graphs formed by adding an edge to $H_1$.
These six graphs generate families of size 6, 2, 2, and 1. Three of the six
graphs are in the family of size 6 and there is one each in 
the remaining three families.

All graphs in the family of size six are ancestors of three graphs. 
In Figure~\ref{fig:sixfam} we provide knotless embeddings of those three
graphs. By Lemma~\ref{lem:tyyt}, all graphs in this family are nIK, hence not MMIK.

\begin{figure}[htb]
\centering
\includegraphics[scale=1]{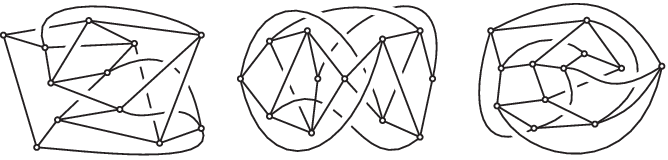}
\caption{Knotless embeddings of three graphs in the size six family from $H_1$.} 
\label{fig:sixfam}
\end{figure}

In a family of two graphs, there is a single $\ty$ move. In Figure~\ref{fig:two2s}
we give knotless embeddings of the children in these two families. 
By Lemma~\ref{lem:tyyt}, all graphs in these two families are nIK, hence not MMIK.

\begin{figure}[htb]
\centering
\includegraphics[scale=1]{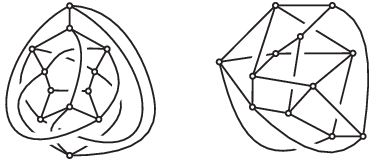}
\caption{Knotless embeddings of two graphs in the size two families from $H_1$.} 
\label{fig:two2s}
\end{figure}

The unique graph in the family of size one is $G_1$. 
In subsection~\ref{sec:3MMIKpf}
we show that this graph is MMIK.
Using the edge list of $G_1$ given above near the beginning of Section~3:
    $$[(0, 4), (0, 5), (0, 9), (0, 10), (1, 4), (1, 6), (1, 7), (1, 10), (2, 3), (2, 4), (2, 5), (2, 9),$$
    $$ (2, 10), (3, 6), (3, 7), (3, 8), (4, 8), (5, 6), (5, 7), (5, 8), (6, 9), (7, 9), (8, 10)],$$
we recover $H_1$ by deleting edge $(2,5)$.

Again, since $H_2$ is $4$-regular, MMIK expansions formed by vertex splits (if any) 
will be in the families of $H_2+e$ graphs. Up to isomorphism, 
there are three $H_2+e$ graphs. These produce a family of size four and another
of size two. 

\begin{figure}[htb]
\centering
\includegraphics[scale=1]{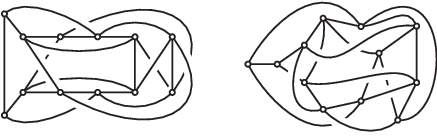}
\caption{Knotless embeddings of two graphs in the size four family from $H_2$.} 
\label{fig:fourfam}
\end{figure}

The family of size four includes two $H_2+e$ graphs. All graphs in the 
family are ancestors of the two graphs that are each shown
to have a knotless embedding in 
Figure~\ref{fig:fourfam}. By Lemma~\ref{lem:tyyt}, all graphs in this family
are nIK, hence not MMIK.

The family of size two consists of the graphs $G_2$ and $G_3$. In 
subsection~\ref{sec:3MMIKpf} we show that these two graphs are MMIK. 
Using the edge list for $G_2$ given above near the beginning of 
Section~3:
   $$[(0, 2), (0, 4), (0, 5), (0, 6), (0, 7), (1, 5), (1, 6), (1, 7), (1, 8), (2, 6), (2, 7), (2, 8),$$
    $$(2, 9), (3, 7), (3, 8), (3, 9), (3, 10), (4, 8), (4, 9), (4, 10), (5, 9), (5, 10), (6, 10)],$$
we recover $H_2$ by deleting edge $(0,2)$.   
As for $G_3$:
    $$[(0, 4), (0, 5), (0, 7), (0, 11), (1, 5), (1, 6), (1, 7), (1, 8), (2, 7), (2, 8), (2, 9), (2, 11),$$
    $$(3, 7), (3, 8), (3, 9), (3, 10), (4, 8), (4, 9), (4, 10), (5, 9), (5, 10), (6, 10), (6, 11)],$$
contracting edge $(6,11)$ leads back to $H_2$.

\section{
\label{sec:ord10}%
Knotless embedding obstructions of order ten.}

In this section, we prove Theorem~\ref{thm:ord10}: there are exactly 35 obstructions to knotless 
embedding of order ten. As in the previous
section, we refer to knotless embedding obstructions as MMIK graphs.
We first describe the 26 graphs given in~\cite{FMMNN,MNPP}
and then list the 9 new graphs unearthed by our computer search.

\subsection{26 previously known order ten MMIK graphs.}
\label{sec:26known}

In~\cite{FMMNN}, the authors describe 264 MMIK graphs. There are three
sporadic graphs (none of order ten), the rest falling into four graph families. 
Of these, 24 have ten vertices and they appear in the families
as follows.

There are three MMIK graphs of order ten in the Heawood family~\cite{KS, GMN, HNTY}:  $H_{10}, F_{10}, $ and $E_{10}$.

In~\cite{GMN}, the authors study the other three families.
All 56 graphs in the $K_{3,3,1,1}$ family are MMIK. 
Of these, 11 have order ten: Cousins 4, 5, 6, 7, 22, 25, 26, 27, 28, 48, and 51. There are 33 MMIK graphs in the family of $E_9+e$. Of these seven have order ten: Cousins 3, 28, 31, 41, 44, 47, and 50. Finally, the family of $G_{9,28}$ includes 156 MMIK graphs. Of these, there are three of
order ten: Cousins 2, 3, and 4.

The other two known MMIK graphs of order ten are described in~\cite{MNPP}, 
one having size 26 and the other size 30.
We remark that the family for the graph of size 26 includes  
both MMIK graphs and graphs with $\delta(G) = 2$. However, 
no ancestor or descendant of a $\delta(G) = 2$ graph is MMIK. 
This is part of our motivation for Question~\ref{que:d3MMIK}


\subsection{Nine new MMIK graphs of order ten}

In this subsection we list the nine additional MMIK graphs that 
we found after a computer search described following the list.
In each case, we use the program
of~\cite{MN} to verify that the graph we found is IK. 
We use the Mathematica implementation of the program 
available at Ramin Naimi's website~\cite{NW}.
To show that the graph is MMIK, we must in addition verify that 
each minor formed by deleting or contracting an edge is nIK. 
Many of these minors are $2$-apex and not IK by Lemma~\ref{lem:2apex}.
There remain 21 minors and below we discuss how we know that those are 
also nIK.

First we list the nine new MMIK graphs of order ten, 
including size, graph6 format~\cite{sage}, and an edge list.

\begin{enumerate}
    \item Size: 25; graph6 format: \verb"ICrfbp{No"
    $$[(0, 3), (0, 4), (0, 5), (0, 6), (1, 4), (1, 5), (1, 6), (1, 7), (1, 8),$$
    $$(2, 5), (2, 6), (2, 7), (2, 8), (2, 9), (3, 6), (3, 7), (3, 8),$$
    $$(3, 9), (4, 7), (4, 8), (4, 9), (5, 8), (5, 9), (6, 9), (7, 9)]$$
    
    \item Size: 25; graph6 format: \verb"ICrbrrqNg"
    $$[(0, 3), (0, 4), (0, 5), (0, 8), (1, 4), (1, 5), (1, 6), (1, 7), (1, 8),$$
    $$(2, 5), (2, 6), (2, 7), (2, 8), (2, 9), (3, 6), (3, 7), (3, 8),$$
    $$(3, 9), (4, 6), (4, 7), (4, 9), (5, 9), (6, 8), (6, 9), (8, 9)]$$
    
    \item Size: 25; graph6 format: \verb"ICrbrriVg"
    $$[(0, 3), (0, 4), (0, 5), (0, 8), (1, 4), (1, 5), (1, 6), (1, 7), (1, 8),$$
    $$(1, 9), (2, 5), (2, 6), (2, 7), (2, 8), (3, 6), (3, 7), (3, 9),$$
    $$(4, 6), (4, 7), (4, 8), (4, 9), (5, 9), (6, 8), (6, 9), (8, 9)]$$
    
    \item Size: 25; graph6 format: \verb"ICrbrriNW"
    $$[(0, 3), (0, 4), (0, 5), (0, 8), (1, 4), (1, 5), (1, 6), (1, 7), (1, 8),$$
    $$(2, 5), (2, 6), (2, 7), (2, 8), (2, 9), (3, 6), (3, 7), (3, 9),$$
    $$(4, 6), (4, 7), (4, 8), (4, 9), (5, 9), (6, 8), (7, 9), (8, 9)]$$
    
    
    \item Size: 27; graph6 format: \verb"ICfvRzwfo"
    $$[(0, 3), (0, 4), (0, 5), (0, 6), (0, 8), (0, 9), (1, 5), (1, 6), (1, 7),$$
    $$(1, 8), (2, 5), (2, 6), (2, 7), (2, 8), (3, 4), (3, 5), (3, 7), (3, 8),$$
    $$(3, 9), (4, 6), (4, 7), (4, 8), (4, 9), (5, 7), (5, 9), (6, 9), (7, 9)]$$
    
    \item Size: 29; graph6 format: \verb"ICfvRr^vo"
    $$[(0, 3), (0, 4), (0, 5), (0, 6), (0, 8), (0, 9), (1, 5), (1, 6), (1, 7), (1, 8),$$
    $$(1, 9), (2, 5), (2, 6), (2, 7), (3, 4), (3, 5), (3, 7), (3, 8), (3, 9), (4, 6),$$
    $$(4, 7), (4, 8), (4, 9), (5, 8), (5, 9), (6, 8), (6, 9), (7, 8), (7, 9)]$$

    
    \item Size: 30; graph6 format: \verb"IQjuvrm^o"
    $$[(0, 2), (0, 4), (0, 5), (0, 6), (0, 7), (0, 8), (1, 3), (1, 5), (1, 6), (1, 7),$$
    $$(1, 8), (1, 9), (2, 4), (2, 5), (2, 7), (2, 8), (2, 9), (3, 5), (3, 6), (3, 7),$$
    $$(3, 9), (4, 6), (4, 7), (4, 8), (4, 9), (5, 8), (5, 9), (6, 8), (6, 9), (7, 9)]$$
    
    \item Size: 31; graph6 format: \verb"IQjur~m^o"
    $$[(0, 2), (0, 4), (0, 5), (0, 6), (0, 8), (1, 3), (1, 5), (1, 6), (1, 7), (1, 8), (1, 9),$$
    $$(2, 4), (2, 5), (2, 7), (2, 8), (2, 9), (3, 5), (3, 6), (3, 7), (3, 9), (4, 6),$$
    $$(4, 7), (4, 8), (4, 9), (5, 7), (5, 8), (5, 9), (6, 7), (6, 8), (6, 9), (7, 9)]$$
    
    \item Size: 32; graph6 format: \verb"IEznfvm|o"
    $$[(0, 3), (0, 4), (0, 5), (0, 6), (0, 7), (0, 8), (0, 9), (1, 3), (1, 4), (1, 5), (1, 6),$$
    $$(1, 7), (1, 8), (1, 9), (2, 4), (2, 5), (2, 6), (2, 7), (2, 8), (2, 9), (3, 6), (3, 7),$$
    $$(3, 9), (4, 5), (4, 7), (4, 8), (5, 8), (5, 9), (6, 7), (6, 8), (6, 9), (7, 9)]$$
    
\end{enumerate}

To complete our argument, it remains to argue that the 21 non 2-apex minors are
nIK. Each of these minors is formed by deleting or contracting an edge 
in one of the nine graphs just listed. Of these minors, 19 have a 
2-apex descendant and are nIK by Lemmas~\ref{lem:2apex} and \ref{lem:tyyt}. In Figure~\ref{fig:ord10}
we give knotless embeddings of the remaining two minors showing 
that they are also nIK.

\begin{figure}[htb]
\centering
\includegraphics[scale=1]{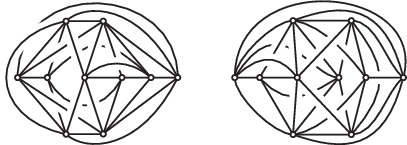}
\caption{Knotless embeddings of two order ten graphs.} 
\label{fig:ord10}
\end{figure}

Let us describe our computer search. A MMIK graph $G$ of order ten must be connected, have 
$\delta(G) \leq 3$ (by Lemma~\ref{lem:delta3}) and have size (number of edges) between 21 and 35.
For, the lower bound on size, see~\cite{M}. The upper bound follows as Mader~\cite{Ma} showed
that a graph with $n \geq 7$ vertices and $5n-14$ or more edges has a $K_7$ minor. 
Since $K_7$ is IK (see \cite{CG}), this means
that a graph of order ten with 36 or more edges cannot be MMIK as it has a proper IK minor, $K_7$.
There remain just under 5 million graphs to consider.

We next sieve out any graph that has one of the 26 known MMIK graphs of order ten (see subsection~\ref{sec:26known})
as a subgraph or any of the known MMIK graphs of order less than ten as a minor. 
We also discard any $2$-apex graph, which must be nIK by Lemma~\ref{lem:2apex}.
We test the remaining graphs using Ramin Naimi's implementation~\cite{NW} of the algorithm
of Miller and Naimi~\cite{MN}. Those which were found to be IK led us to the nine new MMIK graphs
listed above. Those graphs that we could not otherwise classify led us to a list of 35 graphs that we subsequently
showed to be maxnik: following~\cite{EFM}, we say that a graph $G$ is {\em maximal knotless} or {\em maxnik} 
if $G$ is nIK, but every $G+e$ is IK. In Appendix~\ref{sec:appmnik} we provide a classification 
of the 49 maxnik graphs of order ten. In addition to the 35 graphs just mentioned, there are 14 maximal $2$-apex graphs.

Once we determined the 35 MMIK obstructions and the 35 maxnik (and not $2$-apex) graphs of order ten, 
we were in a position to show that the remaining nearly 5 million candidates $G$ are not MMIK for 
at least one of the following three reasons:
1) $G$ is $2$-apex and nIK by Lemma~\ref{lem:2apex}, 
2) $G$ has a proper minor that is IK and is therefore 
not MMIK, or 3) $G$ is a subgraph of one of the 35 maxnik graphs and is therefore nIK.

Having determined the 49 maxnik graphs of order ten (see Appendix~\ref{sec:appmnik}), we can update 
some observations of \cite{EFM}, which catalogs the maxnik graphs through order nine. The fourteen maximally $2$-apex
graphs are of the form $K_2 \ast T_{8}$, the join of $K_2$ and a planar triangulation on eight vertices. As 
such, each has $35$ edges and $\Delta(G) = 9$. The remaining 35 maxnik graphs range between 
size 23 and 34. In particular, we still know of no maxnik graph on 22 edges. 
To continue Tables 1 and 2 of \cite{EFM}, the least $|E|/|V|$ among maxnik graphs of 
order ten is $23/10$ and we have $2 \leq \delta(G) \leq 6$ and $5 \leq \Delta(G) \leq 9$.
Although we still have no regular maxnik graph, there are several examples
where $\Delta(G) - \delta(G)$ is only one.

\section*{Acknowledgements}
We thank Ramin Naimi for use of his programs that were 
essential for this project. We thank the referees for a careful reading 
of an earlier version of this paper;
their feedback resulted in substantial improvements.

\appendix

\section{Proof by hand that $G_1$ and $G_2$ are IK}\label{sec:appG12}

In this section, we give a traditional proof (ie one that does not rely on computers) that $G_1$ and $G_2$ are IK. 
For this we use a lemma due, independently, to two groups~\cite{F,TY}. 
Let $\DD$ denote the multigraph of Figure~\ref{fig:D4} and, for $ i = 1,2,3,4$, let $C_i$ be the cycle
of edges $e_{2i-1}, e_{2i}$. For any given embedding of $\DD$, let $\sigma$
denote the mod 2 sum of the Arf invariants of the 16 Hamiltonian cycles in $\DD$
and $\mbox{lk}(C,D)$ the mod 2 linking number of cycles $C$ and $D$.
Since the Arf invariant of the unknot is zero, an embedding of $\DD$ with
$\sigma \neq 0$ must have a knotted cycle.

\begin{figure}[htb]
\centering
\includegraphics[scale=1]{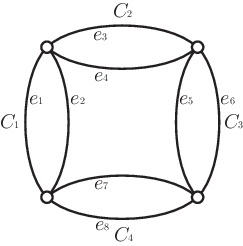}
\caption{The $\DD$ graph.} 
\label{fig:D4}
\end{figure}

\begin{lemma}\cite{F,TY}
\label{lem:D4}
Given an embedding of $\DD$, $\sigma \neq 0$ if and only if
$\mbox{lk}(C_1,C_3) \neq 0$ and $\mbox{lk}(C_2,C_4) \neq 0$.
\end{lemma}

\subsection{$G_1$ is IK}\ 

In this subsection, we show that $G_1$ is IK. 
To use Lemma~\ref{lem:D4} we need pairs of linked cycles. 
We first describe three sets of pairs that we will call $A_i$'s, $B_i$'s, and $C_i$'s

\smallskip
\noindent%
{\bf Step I:} Define $A_i$ pairs.

To find pairs of linked cycles, we use
minors of $G_1$ that are members of the Petersen family of graphs
as these are the obstructions to linkless embedding~\cite{RST}. 
Our first example is
based on
contracting the edges $(3,6)$, $(5,7)$, and $(6,9)$
in Figure~\ref{fig:G1} 
to produce a $K_{4,4}$ minor with 
$\{3,4,5,10\}$ and $\{0,1,2,8\}$ as the two parts. For convenience, we
use the smallest vertex label to denote the new vertex obtained when contracting
edges. Thus, we denote by 3 the vertex obtained by identifying 3, 6, and 9 
of $G_1$. Further deleting the edge $(1,3)$ we identify the Petersen 
family graph $K_{4.4}^-$ as a minor of $G_1$. 

\begin{figure}[htb]
\centering
\includegraphics[scale=1.25]{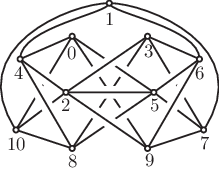}
\caption{The graph $G_1$.} 
\label{fig:G1}
\end{figure}

There are nine pairs of disjoint cycles in 
$K_{4,4}^-$ and we denote these pairs as $A_1$ through $A_9$. 
In Table~\ref{tab:An}, we first give the cycle pair in the $K_{4,4}^-$ 
and then the corresponding pair in $G_1$.

\begin{table}[htb]
    \centering
    \begin{tabular}{c|l|l}
         $A_1$ & 0,4,1,5 -- 2,3,8,10 & 0,4,1,7,5 -- 2,3,8,10 \\
         $A_2$ & 0,4,1,10 -- 2,3,8,5 & 0,4,1,10 -- 2,3,8,5 \\
         $A_3$ & 1,4,2,5 -- 0,3,8,10 & 1,4,2,5,7 -- 0,9,6,3,8,10 \\
         $A_4$ & 1,4,2,10 -- 0,3,8,5 & 1,4,2,10 -- 0,9,6,3,8,5 \\
         $A_5$ & 1,4,8,5 -- 0,3,2,10 & 1,4,8,5,7 -- 0,9,6,3,2,10 \\
         $A_6$ & 1,4,8,10 -- 0,3,2,5 & 1,4,8,10 -- 0,9,6,3,2,5 \\
         $A_7$ & 1,5,2,10 -- 0,3,8,4 & 1,7,5,2,10 -- 0,9,6,3,8,4 \\
         $A_8$ & 0,5,1,10 -- 2,3,8,4 & 0,5,7,1,10 -- 2,3,8,4 \\
         $A_9$ & 1,5,8,10 -- 0,3,2,4 & 1,7,5,8,10 -- 0,9,6,3,2,4 \\
    \end{tabular}
    \caption{Nine pairs of cycles in $G_1$ called $A_1, \ldots, A_9$.}
    \label{tab:An}
\end{table}

\smallskip
\noindent%
{\bf Step II:} Find 
pairs $B_i$ and $C_i$.

Similarly, we will describe a $K_{3,3,1}$ minor that gives pairs of cycles
$B_1$ through $B_9$. Contract edges 
$(1,7)$, $(3,7)$, and $(7,9)$. Delete vertex $6$ and edge $(2,9)$.
The result is a $K_{3,3,1}$ minor with parts $\{0,2,8\}$, $\{4,5,10\}$, and
$\{1\}$. In Table~\ref{tab:Bn} we give the nine pairs of cycles, first
in $K_{3,3,1}$ and then in $G_1$.

\begin{table}[htb]
    \centering
    \begin{tabular}{c|l|l}
         $B_1$ & 0,1,4 -- 2,5,8,10 & 0,4,1,7,9 -- 2,5,8,10 \\
         $B_2$ & 0,1,5 -- 2,4,8,10 & 0,5,7,9 -- 2,4,8,10 \\
         $B_3$ & 0,1,10 -- 2,4,8,5 & 0,9,7,1,10 -- 2,4,8,5 \\
         $B_4$ & 1,2,4 -- 0,5,8,10 & 1,4,2,3,7 -- 0,5,8,10 \\
         $B_5$ & 1,2,5 -- 0,4,8,10 & 2,3,7,5 -- 0,4,8,10 \\
         $B_6$ & 1,2,10 -- 0,4,8,5 & 2,3,7,1,10 -- 0,4,8,5 \\
         $B_7$ & 1,4,8 -- 0,5,2,10 & 1,4,8,3,7 -- 0,5,2,10 \\
         $B_8$ & 1,5,8 -- 0,4,2,10 & 3,7,5,8 -- 0,4,2,10 \\
         $B_9$ & 1,8,10 -- 0,4,2,5 & 1,7,3,8,10 -- 0,4,2,5 \\
    \end{tabular}
    \caption{Nine pairs of cycles in $G_1$ called $B_1, \ldots, B_9$.}
    \label{tab:Bn}
\end{table}

Another $K_{4,4}^-$ minor of $G_1$ will give our last set of nine cycle pairs.
Contract edges $(0,9)$, $(2,5)$, and $(3,8)$ to obtain a $K_{4,4}$
with parts $\{0,1,2,8\}$ and $\{4,6,7,10\}$. Then delete
edge $(1,4)$ to make a $K_{4.4}^-$ minor. Table~\ref{tab:Cn} lists
the nine pairs of cycles, first in the $K_{4,4}^-$ minor and
then in $G_1$.

\begin{table}[htb]
    \centering
    \begin{tabular}{c|l|l}
         $C_1$ & 0,6,1,7 -- 2,4,8,10 & 9,6,1,7 -- 2,4,8,10 \\
         $C_2$ & 0,6,1,10 -- 2,4,8,7 & 0,9,6,1,10 -- 2,4,8,3,7,5 \\
         $C_3$ & 1,6,2,7 -- 0,4,8,10 & 1,6,5,7 -- 0,4,8,10 \\
         $C_4$ & 1,6,2,10 -- 0,4,8,7 & 1,6,5,2,10 -- 0,4,8,3,7,9 \\
         $C_5$ & 1,6,8,7 -- 0,4,2,10 & 1,6,,3,7 -- 0,4,2,10 \\
         $C_6$ & 1,6,8,10 -- 0,4,2,7 & 1,6,3,8,10 -- 0,4,2,5,7,9 \\
         $C_7$ & 0,7,1,10 -- 2,4,8,6 & 0,9,7,1,10 -- 2,4,8,3,6,5 \\
         $C_8$ & 1,7,2,10 -- 0,4,8,6 & 1,7,5,2,10 -- 0,4,8,3,6,9 \\
         $C_9$ & 1,7,8,10 -- 0,4,2,6 & 1,7,3,8,10 -- 0,4,2,5,6,9 \\
    \end{tabular}
    \caption{Nine pairs of cycles in $G_1$ called $C_1, \ldots, C_9$.}
    \label{tab:Cn}
\end{table}

As shown by Sachs~\cite{S}, in any embedding of $K_{4,4}^-$ 
or $K_{3,3,1}$, at least one
pair of the nine disjoint cycles in each graph has odd linking number.
We will simply say the cycles are {\em linked} if the linking number is odd.
Fix an embedding of $G_1$. Our goal is to show that the embedding must
have a knotted cycle.

We will argue by contradiction. For a contradiction, assume that there 
is no knotted cycle in the embedding of $G_1$. We leverage Lemma~\ref{lem:D4}
to deduce that certain pairs of cycles are not linked. Eventually, 
we will conclude that none of $B_1, \ldots, B_9$ are linked. 
This is a contradiction as these correspond to cycles in a $K_{3,3,1}$
and we know that every embedding of this Petersen family graph
must have a pair of linked cycles~\cite{S}. 
The contradiction shows that there must in fact be a knotted 
cycle in the embedding of $G_1$. As the embedding is arbitrary, this
shows that $G_1$ is IK.

\smallskip
\noindent%
{\bf Step III:} Eliminate
$A_2$ by combining with each $B_i$.

We illustrate our strategy by first focusing on the pair 
$A_2 = $ 0,4,1,10  -- 2,3,8,5.
Combine $A_2$ with each $B_i$. In each case we form a $\DD$ graph
as in Figure~\ref{fig:D4}. Since the $B_i$ are pairs of cycles in $K_{3,3,1}$,
a Petersen family graph, at least one pair is linked~\cite{S}.
If $A_2$ is 
also linked, then Lemma~\ref{lem:D4} implies that the embedding of
$G_1$ has a knotted cycle, in contradiction to our assumption.
Therefore, we conclude that $A_2$ is not linked (i.e., does not 
have odd linking number).

\begin{table}[htb]
    \centering
    \begin{tabular}{c|l}
         $B_1$ & $\{0,1,4\}$, $\{2,5,8\}$, $\{3,7\}$, $\{10\}$ \\
         $B_2$ & $\{0\}$, $\{1,4,10\}$, $\{2,3,8\}$, $\{5\}$ \\
         $B_3$ & $\{0,1,10\}$, $\{2,5,8\}$, $\{3,7\}$, $\{4\}$ \\
         $B_4$ & $\{0,10\}$, $\{1,4\}$, $\{2,3\}$, $\{5,8\}$ \\
         $B_5$ & $\{0,4,10\}$, $\{1\}$, $\{2,3,5\}$, $\{8\}$ \\ 
         $B_6$ & $\{0,4\}$, $\{1,10\}$, $\{2,3\}$, $\{5,8\}$ \\
         $B_7$ & $\{0,10\}$, $\{1,4\}$, $\{2,5\}$, $\{3,8\}$ \\
         $B_8$ & $\{0,4,10\}$, $\{1,7\}$, $\{2\}$, $\{3,8,5\}$ \\
         $B_9$ & $\{0,4\}$, $\{1,10\}$, $\{2,5\}$, $\{3,8\}$ \\
    \end{tabular}
    \caption{Pairing $A_2$ with $B_1, \ldots, B_9$.}
    \label{tab:A2}
\end{table}

In Table~\ref{tab:A2}, we list the vertices in $G_1$ that are identified to give each of the four vertices of $\DD$. 
Let us examine the pairing with $B_2$ as an example 
to see how this results in a $\DD$.
We identify $\{1,4,10\}$ as a single vertex by contracting edges $(1,4)$ and
$(1,10)$. Similarly contract $(2,3)$ and $(3,8)$ to make a vertex of
the $\DD$ from vertices $\{2,3,8\}$ of $G_1$. In this way, the cycle 0,4,1,10
of $A_2$ in $G_1$ becomes cycle $C_1$ of the $\DD$ 
(see Figure~\ref{fig:D4})
between $\{1,4,10\}$ and $\{0\}$ and
the cycle 2,3,8,5 becomes cycle $C_3$ between $\{2,3,8\}$ and $\{5\}$.
Similarly 0,5,7,9 of $B_2$ becomes homeomorphic to the cycle $C_2$
between $\{0\}$ and $\{5\}$. For the final cycle of $B_2$, 2,4,8,10, 
we observe that, in homology, 
$\mbox{2,4,8,10 } = \mbox{ 1,4,2,10 } \cup  \mbox{ 1,4,8,10}$.
Assuming $\mbox{lk}((\mbox{0,5,7,9}),(\mbox{2,4,8,10})) \neq 0$,
then one of 
$\mbox{lk}((\mbox{0,5,7,9}),(\mbox{1,4,2,10}))$ and
$\mbox{lk}((\mbox{0,5,7,9}),(\mbox{1,4,8,10}))$ is 
also nonzero. Whichever it is, 1,4,2,10 or 1,4,8,10, that will be our $C_4$ cycle in the $\DD$ of Figure~\ref{fig:D4}.

To summarize, we have argued that $A_2$ forms a $\DD$ 
with each pair $B_1, \ldots, B_9$. Since at least one of the $B_i$'s is
linked, then, assuming $A_2$ is linked, these two pairs make a 
$\DD$ that has a knotted cycle. Therefore,
by way of contradiction, going forward, we may assume $A_2$ is
not linked.

\smallskip
\noindent%
{\bf Step IV:} 
Argue $A_6$ is not linked.

We next argue that $A_6$ is not linked. 
For a contradiction, assume instead that $A_6$ is linked. 
Pairing with the $B_i$'s again, the vertices for each $\DD$ are 
\begin{align*}
    B_1  & \{0,9\}, \{1,4\}, \{2,5\}, \{8,9\} & 
    B_2  & \{0,5,9\}, \{1,7\}, \{2\}, \{4,8,10 \} \\
    B_3  & \{0,9\}, \{1,10\}, \{2,5\}, \{4,8\} & 
    B_4  & \{0,5\}, \{1,4\}, \{2,3\}, \{8,10 \} \\
    B_5  & \{0\}, \{1,7\}, \{2,3,5\}, \{4,8,10\} & 
    B_6  & \{0,5\}, \{1,10\}, \{2,3\}, \{4,8 \} \\
    B_7  & \{0,2,5\}, \{1,4,8\}, \{3\}, \{10\} & 
    B_9  & \{0,2,5\}, \{1,8,10\}, \{3\}, \{4 \}.\\
\end{align*}
For $B_8 = $ 3,7,5,8 -- 0,4,2,10, we first split one of the $A_6$
cycles: 
$\mbox{0,5,2,3,6,9 }  = \mbox{ 0,5,2,9 } \cup  \mbox{ 2,3,6,9}$.
One of the two summands must link with the other $A_2$ cycle
$1,4,8,10$. If 
$\mbox{lk}((\mbox{0,5,2,9}),(\mbox{1,4,8,10})) \neq 0$, 
then, by a symmetry of $G_1$, 
$\mbox{lk}((\mbox{8,5,2,3}),(\mbox{1,4,0,10})) \neq 0$. But this last is 
the pair $A_2$, and we have already argued that $A_2$ is not linked. 

Therefore, it must be that 
$\mbox{lk}((\mbox{2,3,6,9}),(\mbox{1,4,8,10})) \neq 0$.
Next, we split a cycle of $B_8$: 
$\mbox{0,4,2,10 }  = \mbox{ 1,4,2,10 } \cup  \mbox{ 0,4,1,10}$.
If $\mbox{lk}((\mbox{3,7,5,8}),(\mbox{1,4,2,10})) \neq 0$, 
form a $\DD$ with vertices $\{1,4,10\}$, $\{2\}$, $\{3\}$, and $\{8\}$.
On the other hand, 
If $\mbox{lk}((\mbox{3,7,5,8}),(\mbox{0,4,1,10})) \neq 0$, 
form a $\DD$ with vertices $\{0,9\}$, $\{1,4,10\}$, $\{3\}$,
and $\{8\}$.

For every choice of $B_i$ we can make a $\DD$ with $A_6$. We know
that at least one $B_i$ is a linked pair. If $A_6$ is also linked,
then, by Lemma~\ref{lem:D4}, this embedding of $G_1$ has a knotted 
cycle. Therefore, by way of contradiction, we may assume the 
pair of $A_6$ is not linked.

\smallskip
\noindent%
{\bf Step V:} Eliminate $B_2$ and $B_8$.

We next eliminate $B_2$ by pairing it with each $A_i$. As we are assuming
$A_2$ and $A_6$ are not linked, it must be some other $A_i$ pair that is 
linked. Here are the vertices of the $\DD$'s in each case:
\begin{align*}
    A_1  & \{0,5,7\}, \{2,8,10\}, \{3,6,9\}, \{4\} & 
    A_3  & \{0,9\}, \{2,4\}, \{5,7\}, \{8,10 \} \\
    A_4  & \{0,5,9\}, \{1,7\}, \{2,4,10\}, \{8\} & 
    A_5  & \{0,9\}, \{2,10\}, \{4,8\}, \{5,7 \} \\
    A_7  & \{0,9\}, \{2,10\}, \{4,8\}, \{5,7\} & 
    A_8  & \{0,5,7\}, \{2,4,8\}, \{3,6,9\}, \{10\} \\
    A_9  & \{0,9\}, \{2,4\}, \{5,7\}, \{8,10\}. \\
\end{align*}
This shows $B_2$ is not linked. Since $B_8$ is 
the same as $B_2$ by a symmetry of $G_1$, $B_8$
is likewise not linked. Ultimately, we will show 
that no $B_i$ is linked. So far, we have this for $B_2$ and $B_8$.

\smallskip
\noindent%
{\bf Step VI:} Eliminate $C_1$, $C_5$, and $C_3$.

Our next step is to argue $C_1$ is not linked by pairing it 
with the remaining $A_i$'s:
\begin{align*}
    A_1  & \{1,7\}, \{2,8,10\}, \{3,6\}, \{4\} & 
    A_3  & \{1,7\}, \{2,4\}, \{6,9\}, \{8,10 \} \\
    A_4  & \{1\}, \{2,4,10\}, \{6,9\}, \{8\} & 
    A_5  & \{1,7\}, \{2,10\}, \{4,8\}, \{6,9 \} \\
    A_7  & \{0,4,8\}, \{1,5,7\}, \{6\}, \{10\} & 
    A_8  & \{1,7\}, \{2,4,8\}, \{3,6\}, \{10\} \\
    A_9  & \{1,7\}, \{2,4\}, \{6,9\}, \{8,10\}. \\
\end{align*}
By a symmetry of $G_1$, $C_5$ is also not linked.

Now we argue that $C_3$ is not linked, again by pairing
with $A_i$'s:
\begin{align*}
    A_1  & \{0,4\}, \{1,5,7\}, \{3,6\}, \{8,10\} & 
    A_3  & \{0,8,10\}, \{1,5,7\}, \{4\}, \{6 \} \\
    A_5  & \{0,10\}, \{1,5,7\}, \{4,8\}, \{6\} &
    A_7  & \{0,4,8\}, \{1,5,7\}, \{6\}, \{10\} \\ 
    A_8  & \{0,10\}, \{1,5,7\}, \{4,8\}, \{6\} &
    A_9  & \{0,4\}, \{1,5,7\}, \{6\}, \{8,10\}. \\
\end{align*}
For $A_4$ we split the second cycle:
$\mbox{0,9,6,3,8,5 }  = \mbox{ 0,9,6,5 } \cup  \mbox{ 6,3,8,5}$. 
Suppose that it is 0,9,6,5 that is linked with 1,4,2,10. In this case
we also split that cycle:
$\mbox{1,4,2,10 }  = \mbox{ 1,4,8,10 } \cup  \mbox{ 2,4,8,10}$. 
To get a $\DD$ when pairing 0,9,6,5 -- 1,4,8,10 with $C_3$ we use
vertices $\{0\}$, $\{1\}$, $\{4,8,10\}$, and $\{5,6\}$ and
for 0,9,6,5 -- 2,4,8,10 with $C_3$, $\{0\}$, $\{2,3,7\}$, $\{4,8,10\}$,
and $\{5,6\}$. 

On the other hand, if it is 6,3,8,5 that is linked with
1,4,2,10, we write
$\mbox{1,4,2,10 }  = \mbox{ 0,4,1,10 } \cup  \mbox{ 0,4,2,10}$.
Then when $C_3$ is paired with 6,3,8,5 --- 0,4,1,10, we have 
a $\DD$ using vertices $\{0,4,10\}$, $\{1\}$, $\{5,6\}$, and $\{8\}$
while if $C_3$ is paired with 6,3,8,5 -- 0,4,2,10 the vertices
are $\{0,4,10\}$, $\{2,7,9\}$, $\{5,6\}$, and $\{8\}$.
This completes the argument that $C_3$ is not
linked.

\smallskip
\noindent%
{\bf Step VII:} Eliminate $A_8$ and $A_1$.

We will show that $A_8$ is not linked by pairing with the 
remaining $C_i$'s. For $C_8$, the vertices would be $\{0\}$, $\{1,5,7,10\}$,
$\{2\}$, and $\{3,4,8\}$. The remaining cases involve splitting cycles.

For $C_2$ we write
$\mbox{2,4,8,3,7,5 }  = \mbox{ 3,7,5,8 } \cup  \mbox{ 2,4,8,5}$.
If it is 3,7,5,8 -- 0,9,6,1,10 that is linked, we use vertices
$\{0,1,10\}$, $\{2,9\}$, $\{3,8\}$, and $\{5,7\}$ and
if, instead, 2,4,8,5 -- 0,9,6,10 is linked, we have
$\{0,1,10\}$, $\{2,4,8\}$, $\{3,6\}$, and $\{5\}$.

For $C_4$, it is a cycle of $A_8$ that we rewrite:
$\mbox{0,5,7,1,10 }  = \mbox{ 0,9,7,1,10 } \cup  \mbox{ 0,9,7,5}$.
In either case, we use the same vertices: $\{0,7,9\}$, $\{1,5,6,10\}$,
$\{2\}$, and $\{3,4,8\}$.

For $C_6$, 
$\mbox{0,4,2,5,7,9 }  = \mbox{ 0,4,2,9 } \cup  \mbox{ 2,5,7,9}$.
When 0,4,2,9 is linked with 1,6,3,8,10 the vertices are
$\{0,9\}$, $\{1,10\}$, $\{2,4\}$, and $\{3,8\}$
while if 2,5,7,9 links 1,6,3,8,10, we use
$\{1,10\}$, $\{2\}$, $\{3,8\}$, and $\{5,7\}$.

Continuing with $C_7$, 
$\mbox{2,4,8,3,6,5 }  = \mbox{ 3,6,5,8 } \cup  \mbox{ 2,4,8,5}$.
The vertices for 3,6,5,8 -- 0,9,7,1,0 are $\{0,1,7,10\}$,
$\{2,9\}$, $\{3,8\}$, and $\{5\}$
and for 2,4,8,5 -- 0,9,7,1,10, use 
$\{0,1,7,10\}$, $\{2,4,8\}$, $\{3,6,9\}$, and $\{5\}$.

Finally, in the case of $C_9$, write 
$\mbox{0,4,2,5,6,9 }  = \mbox{ 0,4,2,9 } \cup  \mbox{ 2,5,6,9}$.
When 0,4,2,9 -- 1,7,3,8,10 is linked, the vertices are
$\{0\}$, $\{1,7,10\}$, $\{2,4\}$, and $\{3,8\}$
and for 2,5,6,9 -- 1,7,3,8,10 use $\{1,7,10\}$,
$\{2\}$, $\{3,8\}$, and $\{5\}$.
This completes the argument for $A_8$. By a symmetry of $G_1$,
$A_1$ is also not linked. 
In other words, going forward, we will assume it is one of 
$A_3$, $A_4$, $A_5$, $A_7$, or $A_9$ that is linked.

\smallskip
\noindent%
{\bf Step VIII:} Eliminate $B_1$, $B_3$, and $B_5$.

Next, we will argue that $B_1$ is not linked by comparing 
with the remaining $A_i$'s:
\begin{align*}
    A_3  & \{0,9\}, \{1,4,7\}, \{2,5\}, \{8,10\} & 
    A_4  & \{0,9\}, \{1,4\}, \{2,10\}, \{5,8\} \\
    A_5  & \{0,9\}, \{1,4,7\}, \{2,10\}, \{5,8\} &
    A_7  & \{0,4,9\}, \{1,7\}, \{2,5,10\}, \{8\} \\ 
    A_9  & \{0,4,9\}, \{1,7\}, \{2\}, \{5,8,10\}. \\
\end{align*}
By a symmetry of $G_1$ we also assume $B_3$ is not linked.

Now, by pairing with the remaining $A_i$'s, we show $B_5$ is not
linked:
\begin{align*}
    A_3  & \{0,8,10\}, \{2,5,7\}, \{3\}, \{4\} & 
    A_5  & \{0,10\}, \{2,3\}, \{4,8\}, \{5,7\} \\
    A_7  & \{0,4,8\}, \{2,5,7\}, \{3\}, \{10\} &
    A_9  & \{0,4\}, \{2,3\}, \{5,7\}, \{8,10\}. \\ 
\end{align*}
For $A_4$ we employ several splits. First,
$\mbox{0,9,6,3,8,5 }  = \mbox{ 0,9,6,5 } \cup  \mbox{ 6,3,8,5}$.
In the case that 0,9,6,5 -- 1,4,2,10 is linked, write
$\mbox{1,4,2,10 }  = \mbox{ 2,4,8,10 } \cup  \mbox{ 1,4,8,10}$.
Pairing 0,9,6,5 -- 2,4,8,10 with $B_5$, the vertices are
$\{0\}$, $\{2\}$, $\{4,8,10\}$, and $\{5\}$.
If instead it is 0,9,6,5 -- 1,4,8,10 that is linked, we use
$\{0\}$, $\{1,7\}$, $\{4,8,10\}$, and $\{5\}$.

So we assume that 6,3,8,5 -- 1,4,2,10 is linked and
rewrite a $B_5$ cycle:
$\mbox{2,3,7,5 }  = \mbox{ 2,3,6,5 } \cup  \mbox{ 3,6,5,7}$.
In case 2,3,6,5 -- 0,4,8,10 is linked, we make a further 
split:
$\mbox{1,4,2,10 }  = \mbox{ 1,7,9,2,10 } \cup  \mbox{ 1,4,9,2,10}$.
Thus, assuming 2,3,6,5 -- 0,4,8,10 and 1,7,9,2,10 -- 6,3,8,5 are
both linked, we have a $\DD$ with vertices 
$\{2\}$, $\{3,5,6\}$, $\{8\}$, and $\{10\}$. 
If instead it is 2,3,6,5 -- 0,4,8,10 and 1,4,9,2,10 -- 6,3,8,5 that 
are linked, use $\{2\}$, $\{3,5,6\}$, $\{4\}$, and $\{8\}$.
This leaves the case where 3,6,5,7 -- 0,4,8,10 is linked.
We must split a final time:
$\mbox{1,4,2,10 }  = \mbox{ 0,4,1,10 } \cup  \mbox{ 0,4,2,10}$.
If 3,6,5,7 -- 0,4,8,10 and 0,4,1,10 -- 6,3,8,5 are both linked,
the vertices are $\{0,4,10\}$, $\{1,7\}$, $\{3,5,6\}$, and $\{8\}$.
On the other hand, when 3,6,5,7 -- 0,4,8,10 and 0,4,2,10 -- 6,3,8,5
are linked, use $\{0,4,10\}$, $\{2,7,9\}$, $\{3,5,6\}$, and $\{8\}$.
This completes the argument for $B_5$, which we have shown is 
not linked.

\smallskip
\noindent%
{\bf Step IX:} Eliminate $B_9$. Only $B_4$ and $B_6$ remain.

Next we turn to $B_9$ which we compare with the remaining $A_i$'s:
\begin{align*}
    A_3  & \{0\}, \{1,7\}, \{2,4,5\}, \{3,8,10\} & 
    A_4  & \{0,5\}, \{1,10\}, \{2,4\}, \{3,8\} \\
    A_7  & \{0,4\}, \{1,7,10\}, \{2,5\}, \{3,8\} &
    A_9  & \{0,2,4\}, \{1,7,8,10\}, \{3\}, \{5\}. \\ 
\end{align*}
This leaves $A_5$ for which we rewrite:
$\mbox{0,9,6,3,2,10 }  = \mbox{ 0,9,2,10 } \cup  \mbox{ 2,3,6,9}$.
If 0,9,2,10 -- 1,4,8,5,7 is linked, then observe that, by a symmetry of
$G_1$, this is the same as $A_8$, which is not linked. Thus, we can assume
2,3,6,9 -- 1,4,8,5,7 is linked and rewrite:
$\mbox{1,4,8,5,7 }  = \mbox{ 1,4,0,5,7 } \cup  \mbox{ 0,4,8,5}$.
Pairing 2,3,6,9 - 1,4,0,5,7 with $B_9$ gives a $\DD$ on vertices
$\{0,4,5\}$, $\{1,7\}$, $\{2\}$, and $\{3\}$.
If it is 2,3,6,9 -- 0,4,8,5 that is linked, then, pairing with $B_9$,
use $\{0,4,5\}$, $\{2\}$, $\{3\}$, and $\{8\}$. This
completes the argument for $B_9$ and, by a symmetry of $G_1$, also 
for $B_7$.

Recall that our goal is to argue, for a contradiction, that no
$B_i$ is linked. At this stage we are left only with $B_4$ and $B_6$
as pairs that could be linked.

\smallskip
\noindent%
{\bf Step X:} Eliminate $A_4$, $A_5$, and $A_9$ leaving only $A_3$ and $A_7$.

Before providing the argument for the remaining
two $B_i$'s, we first eliminate a few more $A_i$'s, starting with 
$A_4$, which we compare with the two remaining $B_i$'s:
\begin{align*}
    B_4  & \{0,5,8\}, \{1,2,4\}, \{3\}, \{10\} &
    B_6  & \{0,5,8\}, \{1,2,10\}, \{3\}, \{4\}. \\ 
\end{align*}
Next $A_5$, again by pairing with $B_4$ and $B_6$:
\begin{align*}
    B_4  & \{0,10\}, \{1,4,7\}, \{2,3\}, \{5,8\} &
    B_6  & \{0\}, \{1,7\}, \{2,3,10\}, \{4,8,5\}. \\ 
\end{align*}
Since $A_9$ agrees with $A_5$ under a symmetry of $G_1$,
this leaves only $A_3$ and $A_7$ as pairs that may yet be linked 
among the $A_i$'s.

\smallskip
\noindent%
{\bf Step XI:} Eliminate $C_6$ and $C_9$, leaving $C_2$, $C_4$, $C_7$, and $C_8$.

As a penultimate step, we show that $C_6$ is not linked
by comparing with these two remaining $A_i$'s:
\begin{align*}
    A_3  & \{0,9\}, \{1\}, \{2,4,5,7\}, \{3,6,8,10\} &
    A_7  & \{0,4,9\}, \{1,10\}, \{2,5,7\}, \{3,6,8\}. \\ 
\end{align*}
By a symmetry of $G_1$, we can also assume that $C_9$ is not linked.
This leaves only four $C_i$'s that may be linked: $C_2$, $C_4$, $C_7$,
and $C_8$.

\smallskip
\noindent%
{\bf Step XII:} 
Eliminate the remaining two $B_i$'s ($B_6$ and $B_4$) to complete the argument.

Finally, compare $B_6$ with the remaining $C_i$'s:
\begin{align*}
    C_4  & \{0,4,8\}, \{1,2,10\}, \{3,7\}, \{5\} &
    C_8  & \{0,4,8\}, \{1,2,7,10\}, \{3\}, \{5\}. \\ 
\end{align*}
For $C_2$ we rewrite:
$\mbox{2,4,8,3,7,5 }  = \mbox{ 2,3,8,4 } \cup  \mbox{ 2,3,7,5}$.
If 2,3,8,4 -- 0,9,6,1,10 is linked, then pairing with $B_6$, we 
have vertices $\{0\}$, $\{1,10\}$, $\{2,3\}$, and $\{4,8\}$.
On the other hand pairing $B_6$ with 2,3,7,5 -- 0,9,6,1,10 we will
have a $\DD$ using the vertices $\{0\}$, $\{1,10\}$, $\{2,3,7\}$, and
$\{5\}$.

For $C_7$:
$\mbox{2,4,8,3,6,5 }  = \mbox{ 2,3,8,4 } \cup  \mbox{ 2,3,6,5}$.
Then 2,3,8,4 -- 0,9,7,1,10 with $B_6$ gives a $\DD$ for
vertices $\{0\}$, $\{1,7,10\}$, $\{2,3\}$, and $\{4,8\}$.
On the other hand, 2,3,6,5 -- 0,9,7,1,10 pairs with $B_6$
using vertices $\{0\}$, $\{1,7,10\}$, $\{2,3\}$, and $\{5\}$.

This completes the argument for $B_6$.
By a symmetry of $G_1$ we see that $B_4$ is also not linked.

In this way, the assumption that there is no knotted cycle in
the given embedding of $G_1$ forces us to conclude that
no pair $B_1, \ldots, B_9$ is linked. However, these 
correspond to the cycles of a $K_{3,3,1}$. As Sachs~\cite{S} 
has shown, any embedding of $K_{3,3,1}$ 
must have a pair of cycles with odd linking number. 
The contradiction shows that there can
be no such knotless embedding and $G_1$ is IK.
This completes the proof that $G_1$ is MMIK.

\subsection{$G_2$ is MMIK}\ 

In this subsection, we show that $G_2$ is IK.
The argument is similar to that for $G_1$ above. 
To use Lemma~\ref{lem:D4}, we need pairs of linked cycles in $G_2$.

\smallskip
\noindent%
{\bf Step I:} 
Define pairs $A_i$, $B_i$, $C_i$, and $D_i$.

We begin by identifying four ways that the Petersen family graph $P_9$, of order nine, 
appears as a minor of graph $G_2$. 
Using the vertex labelling of Figure~\ref{fig:G2}, on contracting edge $(0,4)$ and deleting vertex 8, 
the resulting graph has $P_9$ as a subgraph. 

\begin{figure}[htb]
\centering
\includegraphics[scale=1.25]{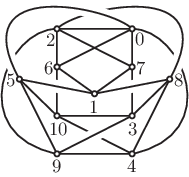}
\caption{The graph $G_2$.} 
\label{fig:G2}
\end{figure}

There are seven pairs of disjoint cycles in $P_9$.
we denote these pairs as $A_1$ through $A_7$.
In Table~\ref{tab:G2An}, we give the pairs in $G_2$.

\begin{table}[htb]
    \centering
    \begin{tabular}{c|l}
         $A_1$ & 0,4,10,5 -- 1,6,2,9,3,7 \\
         $A_2$ & 0,2,6,10,4 -- 1,5,9,3,7 \\
         $A_3$ & 0,2,9,5 -- 1,6,10,3,7 \\
         $A_4$ & 0,5,1,7 -- 2,6,10,3,9 \\
         $A_5$ & 0,4,10,3,7 -- 1,5,9,2,6 \\
         $A_6$ & 1,5,10,6 -- 0,2,9,3,7 \\
         $A_7$ & 3,9,5,10 -- 0,2,6,1,7 \\
    \end{tabular}
    \caption{Seven pairs of cycles in $G_2$ called $A_1, \ldots, A_7$.}
    \label{tab:G2An}
\end{table}

Similarly, if we contract edge $(3,10)$ and delete vertex 7 in $G_2$,
the resulting graph has a $P_9$ subgraph. We will call these seven cycles
$B_1, \ldots B_7$ as in Table~\ref{tab:G2Bn}

\begin{table}[htb]
    \centering
    \begin{tabular}{c|l}
         $B_1$ & 0,2,6 -- 1,5,9,3,10,4,8 \\
         $B_2$ & 0,2,8,4 -- 1,5,9,3,10,6 \\
         $B_3$ & 0,2,9,5 -- 1,6,10,4,8 \\
         $B_4$ & 0,4,10,6 -- 1,5,9,2,8 \\
         $B_5$ & 0,5,1,6 -- 2,8,4,10,3,9 \\
         $B_6$ & 1,6,2,8 -- 0,4,10,3,9,5 \\
         $B_7$ & 2,6,10,3,9 -- 0,4,8,1,5 \\
    \end{tabular}
    \caption{Seven pairs of cycles in $G_2$ called $B_1, \ldots, B_7$.}
    \label{tab:G2Bn}
\end{table}

Contracting edge $(4,10)$ and deleting vertex 6, we have the 
seven cycles of Table~\ref{tab:G2Cn}.

\begin{table}[htb]
    \centering
    \begin{tabular}{c|l}
         $C_1$ & 3,8,4,10 -- 0,2,9,5,1,7 \\
         $C_2$ & 3,8,1,7 -- 0,2,9,5,10,4 \\
         $C_3$ & 2,8,3,9 -- 0,4,10,5,1,7 \\
         $C_4$ & 0,4,10,3,7 -- 1,5,9,2,8 \\
         $C_5$ & 3,9,5,10 -- 0,2,8,1,7 \\
         $C_6$ & 1,5,10,4,8 -- 0,2,9,3,7 \\
         $C_7$ & 0,2,8,4 -- 1,5,9,3,7 \\
    \end{tabular}
    \caption{Seven pairs of cycles in $G_2$ called $C_1, \ldots, C_7$.}
    \label{tab:G2Cn}
\end{table}

Finally, if we contract edge $(3,9)$ and delete vertex 7 in $G_2$,
the resulting graph has a $P_9$ subgraph. We will call these seven cycles
$D_1, \ldots D_7$ as in Table~\ref{tab:G2Dn}.

\begin{table}[htb]
    \centering
    \begin{tabular}{c|l}
         $D_1$ & 2,8,3,9 -- 0,4,10,6,1,5 \\
         $D_2$ & 0,2,9,5 -- 1,6,10,4,8 \\
         $D_3$ & 0,2,8,4 -- 1,5,9,3,10,6 \\
         $D_4$ & 1,5,9,3,8 -- 0,2,6,10,4 \\
         $D_5$ & 1,6,2,8 -- 0,4,10,3,9,5 \\
         $D_6$ & 3,8,4,10 -- 0,2,6,1,5 \\
         $D_7$ & 2,6,10,3,9 -- 0,4,8,1,5 \\
    \end{tabular}
    \caption{Seven pairs of cycles in $G_2$ called $D_1, \ldots, D_7$.}
    \label{tab:G2Dn}
\end{table}

\smallskip
\noindent%
{\bf Step II:} 
Eliminate $A_5$.

We will need to introduce two more Petersen family graph minors later,
but let us begin by ruling out some of the pairs we already have. 
As in our argument for $G_1$, we assume that we have a knotless
embedding of $G_2$ and step by step argue that various 
cycle pairs are not linked (i.e. do not have odd linking number)
using Lemma~\ref{lem:D4}. Eventually, this will allow us
to deduce that all seven pairs $B_1, \ldots, B_7$ are
not linked. This is a contradiction since Sachs~\cite{S}
showed that in any embedding of $P_9$, there
must be a pair of cycles with odd linking number.
The contradiction shows that there is no such knotless embedding
and $G_2$ is IK.

We will see that $A_5$ is not linked by showing it
results in a $\DD$ with every pair $B_1, \ldots, B_7$. 
Indeed the vertices of the $\DD$ are formed by contracting the
following vertices

\begin{align*}
B_1 & \{0\}, \{1,5,9\}, \{2,6\}, \{3,4,10\} &
B_2 & \{0,4\}, \{1,5,6,9\}, \{2\}, \{3,10\} \\
B_3 & \{0\}, \{1,6\}, \{2,5,9\}, \{4,10\} & 
B_4 & \{0,4,10\}, \{1,2,5,9\}, \{3,8\}, \{6\} \\
B_5 & \{0\}, \{1,5,6\}, \{2,9\}, \{3,4,10\} & 
B_7 & \{0,4\}, \{1,5\}, \{2,6,9\},\{3,10\}. \\
\end{align*}
For $B_6 = $ 1,6,2,8 -- 0,4,10,3,9,5, we first split one of the $A_5$
cycles: 
$\mbox{0,4,10,3,7 } = \mbox{0,4,8,3,7 } \cup \mbox{ 3,8,4,10}$.
One of the two summands must link with the other $A_5$ cycle
$1,5,9,2,6$. If 
$\mbox{lk}((\mbox{3,8,4,10}),(\mbox{1,5,9,2,6})) \neq 0$, 
then, by contracting edges, we form a $\DD$ with $A_5$
whose vertices are $\{1,2,6\}, \{3,4,10\}, \{5,9\}, \{8\}$.
On the other hand, if 
$\mbox{lk}((\mbox{0,4,8,3,7}),(\mbox{1,5,9,2,6})) \neq 0$, 
then we will split the $B_6$ cycle
$\mbox{0,4,10,3,9,5 } = \mbox{0,4,10,5 } \cup \mbox{ 3,9,5,10}$.
When 
$\mbox{lk}((\mbox{1,6,2,8}),(\mbox{0,4,10,5})) \neq 0$,
we have a $\DD$ with vertices $\{0,4\}, \{1,2,6\}, \{5\}, \{8\}$
and when
$\mbox{lk}((\mbox{1,6,2,8}),$ $(\mbox{3,9,5,10})) \neq 0$,
the $\DD$ is on $\{1,2,6\}, \{3\}, \{5,9\}, \{8\}$.

We have shown that, for each $B_i$, we must have a $\DD$
with $A_5$. Sachs~\cite{S} showed that in every embedding 
of $P_9$, there is a pair of linked cycles. Thus, in our
embedding of $G_2$, at least one $B_i$ is linked. If $A_5$
were also linked, that would result in a $\DD$ with 
a knotted cycle by Lemma~\ref{lem:D4}. This
contradicts our assumption that we have a knotless embedding
of $G_2$. Therefore, going forward, we can assume $A_5$ is not linked.

\smallskip
\noindent%
{\bf Step III:} 
Eliminate $B_7$ and $A_4$.

We next argue that $B_7$ is not linked by comparing with
$C_1, \ldots, C_7$. For $C_4$, $C_6$, and $C_7$, we 
immediately form a $\DD$ as follows

\begin{align*}
C_4 & \{0,4\}, \{1,5,8\}, \{2,9\}, \{3,7\} &
C_6 & \{0\}, \{1,4,5,8\}, \{2,3,9\}, \{10\} \\
C_7 & \{0,4,8\}, \{1,5\}, \{2\}, \{3,9\}. \\
\end{align*}
For the remaining pairs, we will split a cycle of the $C_i$.
For $C_1$, write
$\mbox{0,2,9,5,1,7 } = \mbox{0,2,7 } \cup \mbox{ 1,5,9,2,7}$.
In the first case,
where 
$\mbox{lk}((\mbox{3,8,4,10}),$ $(\mbox{0,2,7})) \neq 0$,
the $\DD$ is on $\{0\}, \{2\}, \{3,10\}, \{4,8\}$.
In the second case, 
$\mbox{lk}((\mbox{3,8,4,10}),$ $(\mbox{1,5,9,7})) \neq 0$,
we have $\{1,5\}, \{2,9\}, \{3,10\}, \{4,8\}$.

For $C_2$, split
$\mbox{0,2,9,5,10,4 } = \mbox{0,2,9,5 }  \cup \mbox{ 0,4,10,5}$
with, in the first case, a $\DD$ on 
$\{0,5\}, \{1,8\}, \{2,9\}, \{3\}$
and, in the second, on
$\{0,4,5\}, \{1,8\}, \{3\}, \{10\}$.
For $C_5$,
$\mbox{0,2,8,1,7 } = \mbox{0,2,7 } \cup \mbox{ 1,7,2,8}$,
the first case is 
$\{0\}, \{2\}, \{3,9,10\}, \{5\}$
and the second has $\{1,8\}, \{2\}, \{3,9,10\}, \{5\}$.

Finally, for $C_3$
split 
$\mbox{0,4,10,5,1,7 } = \mbox{1,6,7 } \cup \mbox{ 1,6,7,0,4,10,5}$.
In the first case, there is a $\DD$ with vertices
$\{1\}, \{2,3,9\}, \{6\}, \{8\}$.
In the second case, we split the same cycle a second time:
$\mbox{1,6,7,0,4,10,5 } = \mbox{1,5,10,6 } \cup 
\mbox{ 0,4,10,6,7}$.
In the first subcase, we have a $\DD$ with vertices
$ \{1,5\}, \{2,3,9\},\{6,10\}, \{8\}$
and in the second subcase, $\{0,4\}, \{2,3,9\}, \{6,10\}, \{8\}$.

Going forward, we can assume that $B_7$ is not linked.

We next argue that $A_4$ is not linked by comparing with
$D_1, \ldots, D_7$. For four links, we immediately 
give the vertices of the $\DD$:
\begin{align*}
D2 & \{0,5\}, \{1\}, \{2,9\}, \{6,10\} &
D3 & \{0\}, \{1,5\}, \{2\}, \{3,6,9,10\} \\
D4 & \{0\}, \{1,5\}, \{2,6,10\}, \{3,9\} &
D5 & \{0,5\}, \{1\}, \{2,6\}, \{3,9,10\}. \\
\end{align*}

For $D_1$, split the second cycle of $A_4$:
$\mbox{2,6,10,3,9 } = \mbox{3,9,4,10 } \cup \mbox{ 2,9,4,10,6}$.
In the first case, the verices of the $\DD$ are
$\{0,1,5\}, \{2,7\}, \{3,9\}, \{4,10\}$ and 
in the second, $\{0,1,5\}, \{2,9\}, \{3,7\}, \{4,6,10\}$.
For $D_6$, split the second cycle:
$\mbox{0,2,6,1,5 } = \mbox{0,2,9,5 } \cup \mbox{ 1,5,9,2,6}$.
In the first case, we have vertices 
$\{0,5\}, \{1,8\}, \{2,9\}, \{3,10\}$ and in the second,
$\{0,4\}, \{1,5\}, \{2,6,9\}, \{3,10\}$.
Finally, $D_7$ is the same pair of cycles as $B_7$, which 
we have assumed is not linked. Going forward, we will
assume $A_4$ is not linked.

\smallskip
\noindent%
{\bf Step IV:} Introduce pairs $E_i$ to 
eliminate $A_2$ and $A_3$. This leaves only $A_1$, $A_6$, and $A_7$.

We have already argued that we can assume $A_4$ and $A_5$ are 
not linked. In this step we eliminate $A_2$ and $A_3$, leaving
only three $A_i$ that could be linked. For this we use another
Petersen graph minor. 
Using the labelling of Figure~\ref{fig:G2},
partition the vertices as $\{0,8,9,10\}$ and $\{2,3,4,5\}$.
Contract edges $(0,7)$ and $(2,6)$. 
This resulting graph has a $K_{4,4}^-$ 
subgraph, where $(5,8)$ is the missing edge. As in 
Table~\ref{tab:G2En}, we will call the resulting
nine pairs of cycles $E_1, \ldots, E_9$.

\begin{table}[htb]
    \centering
    \begin{tabular}{c|l}
         $E_1$ & 2,8,4,9 -- 0,5,10,3,7 \\
         $E_2$ & 0,4,10,5 -- 2,8,3,9 \\
         $E_3$ & 0,2,6,10,5 -- 3,8,4,9 \\
         $E_4$ & 0,2,4,8 -- 3,9,5,10 \\
         $E_5$ & 4,9,5,10 -- 0,2,8,3,7 \\
         $E_6$ & 0,4,8,3,7 -- 2,6,10,5,9 \\
         $E_7$ & 0,5,9,3,7 -- 2,6,10,4,8 \\
         $E_8$ & 0,4,9,5 -- 2,6,10,3,8 \\
         $E_9$ & 0,2,9,5 -- 3,8,4,10 \\
    \end{tabular}
    \caption{Nine pairs of cycles in $G_2$ called $E_1, \ldots, E_9$.}
    \label{tab:G2En}
\end{table}

We will use $E_1, \ldots, E_9$ to show that $A_2$ may be assumed 
unlinked. 
Except for $E_1$, we list the vertices of the $\DD$:
\begin{align*} 
E_2 & \{0,4,10\}, \{2\}, \{3,9\}, \{5\} & 
E_3 & \{0,2,6,10\}, \{3,9\}, \{4\}, \{5\} \\
E_4 & \{0,2,4\}, \{1,8\}, \{3,5,9\}, \{10\} &
E_5 & \{0,2\}, \{3,7\}, \{4,10\}, \{5,9\} \\
E_6 & \{0,4\}, \{2,6,10\}, \{3,7\}, \{5,9\} & 
E_7 & \{0\}, \{1,8\}, \{2,4,6,10\}, \{3,5,7,9\} \\ 
E_8 & \{0,4\}, \{2,6,10\}, \{3\}, \{5,9\} & 
E_9 &  \{0,2\}, \{3\}, \{4,10\}, \{5,9\}. \\
\end{align*}
The argument for $E_1$ is a little involved. 
Let us split the second cycle
$\mbox{0,5,10,3,7 } = \mbox{0,5,10,6,1,7 } \cup \mbox{ 1,6,10,3,7}$.

\noindent%
Case 1: Suppose that 
$\mbox{lk}((\mbox{2,8,4,9}),(\mbox{0,5,10,6,1,7})) \neq 0$.
Next split the first cycle of $A_2$:
$\mbox{0,2,6,10,4 } = \mbox{0,2,6 } \cup \mbox{ 0,4,10,6}$.

Case 1a): 
Suppose that 
$\mbox{lk}((\mbox{0,2,6}),(\mbox{1,5,9,3,7})) \neq 0$.
We split the second $E_1$ cycle again:
$\mbox{0,5,10,6,1,7 } = \mbox{0,5,10,6 } \cup \mbox{ 0,6,1,7}$
In the first case, where 
$\mbox{lk}((\mbox{2,8,4,9}),(\mbox{0,5,10,6})) \neq 0$,
we have a $\DD$ with vertices
$\{0,6\}, \{2\}, \{5\}, \{9\}$. 
In the second case, 
$\mbox{lk}((\mbox{2,8,4,9}),(\mbox{0,6,1,7})) \neq 0$,
the $\DD$ vertices are $\{0,6\}, \{1,7\},$ $ \{2\}, \{9\}$.

Case 1b):
Suppose that 
$\mbox{lk}((\mbox{0,4,10,6}),(\mbox{1,5,9,3,7})) \neq 0$.
Split the second $E_1$ cycle again:
$\mbox{0,5,10,6,1,7 } = \mbox{0,5,10,6 } \cup \mbox{ 0,6,1,7}$.
In the first case, where
$\mbox{lk}((\mbox{2,8,4,9}),(\mbox{0,5,10,6})) \neq 0$,
the $\DD$ vertices are $\{0,6,10\}, \{4\}, \{5\}, \{9\}$.
In the second case, where 
$\mbox{lk}((\mbox{2,8,4,9}),(\mbox{0,6,1,7})) \neq 0$,
we have vertices $\{0,6\}, \{1,7\},$ $ \{4\}, \{9\}$.

\noindent%
Case 2: 
Suppose that 
$\mbox{lk}((\mbox{2,8,4,9}),(\mbox{1,6,10,3,7})) \neq 0$.
We split the fist cycle of $A_2$:
$\mbox{0,2,6,10,4 } = \mbox{0,2,6 } \cup \mbox{ 0,4,10,6}$.
In the first case, the $\DD$ has vertices
$\{1,3,7\}, \{2\}, \{6\}, \{9\}$ and in the second
$\{1,3,7\}, \{4\}, \{6,10\}, \{9\}$.

This completes the argument for $A_2$, which we henceforth assume
is not linked.

Next we again use $E_1, \ldots, E_9$ to see that $A_3$ is also 
not linked. Except for $E_8$ and $E_9$ we immediately have
a $\DD$:
\begin{align*}
E_1 & \{0,5\}, \{1,8\}, \{2,9\}, \{3,7,10\} &
E_2 & \{0,5\}, \{2,9\}, \{3\}, \{10\} \\
E_3 & \{0,2,5\}, \{3\}, \{6\}, \{9\} & 
E_4 & \{0,2\}, \{1,8\}, \{3,10\}, \{5,9\} \\
E_5 & \{0,2\}, \{3,7\}, \{5,9\}, \{10\} & 
E_6 & \{0\}, \{2,5,9\}, \{3,7\}, \{6,10\} \\
E_7 & \{0,5,9\}, \{2\}, \{3,7\}, \{6,10\} \\
\end{align*}
For $E_8$, split the first cycle
$\mbox{0,4,9,5 } = \mbox{0,5,1,7 } \cup \mbox{ 0,4,9,5,1,7}$.
In the first case, we have a $\DD$ with vertices
$\{0,5\}, \{1,7\}, \{2\}, \{3,6,10\}$.
In the second case, we further split the first cycle of $A_3$:
$\mbox{0,2,9,5 } = \mbox{2,8,4,9 } \cup \mbox{ 0,2,8,4,9,5}$ 
leading to two subcases. In the first subcase the $\DD$ is
$\{1,7\}, \{2,8\}, \{3,6,10\}, \{4,9\}$ and in the second
$\{0,4,5,9\}, \{1,7\}, \{2,8\}, \{3,6,10\}$.

For $E_9$ split the first cycle 
$\mbox{0,2,9,5 } = \mbox{0,2,6,1,5 } \cup \mbox{ 2,6,1,5,9}$
with a $\DD$ on first
$\{0,2,5\}, \{1,6\}, \{3,10\}, \{4,9\}$ 
and then 
$\{0,4\}, \{1,6\}, \{2,5,9\}, \{3,10\}$.

We will not need $E_1, \ldots, E_9$ in the remainder of the argument.
At this stage, we can assume that it is one of $A_1$, $A_6$, and 
$A_7$ that is linked.

\smallskip
\noindent%
{\bf Step V:} Eliminate $B_5$, $B_2$, and $B_3$, leaving only $B_1$, $B_4$, and $B_6$.

Our next step is to argue that we can assume
$B_5$ is not linked by comparing with the three remaining $A_i$'s.
For the first two, we immediately recognize a $\DD$:

\begin{align*}
A_1 & \{0,5\}, \{1,6\}, \{2,3,9\}, \{4,10\} &
A_6 & \{0\}, \{1,5,6\}, \{2,3,9\}, \{10\} \\
\end{align*}

For $A_7$, split the second cycle:
$\mbox{0,2,6,1,7 } =  \mbox{0,2,7 } \cup \mbox{ 2,6,1,7}$
so that the $\DD$ has vertices
$\{0\}, \{2\}, \{3,9,10\}, \{5\}$ in the first case
and then $\{1,6\}, \{2\},$ $ \{3,9,10\}, \{5\}$.
Going forward, we assume that $B_5$ is not linked.

Next we will eliminate $B_2$ and $B_3$. For $B_3$, we have a $\DD$
with each of the remaining $A_i$'s:
\begin{align*}
A_1 & \{0,5\}, \{1,6\}, \{2,9\}, \{4,10\} &
A_6 & \{0,2,9\}, \{1,6,10\}, \{3,8\}, \{5\} \\
A_7 & \{0,2\}, \{1,6\}, \{5,9\}, \{10\} 
\end{align*}
For $B_2$, notice first that, as we are assuming $A_3$ is not linked,
by a symmetry of $G_2$, we can assume that
 0,2,8,4 - 1,6,10,3,7 is also not linked.
Again, since $B_3$ is unlinked, the symmetric pair 
0,2,8,4 - 1,5,9,3,7 is also not linked.
Since $\mbox{1,5,9,3,10,6 } = \mbox{1,5,9,3,7 } \cup \mbox{ 1,6,10,3,7}$
we conclude that $B_2$ is not linked. Having eliminated four
of the $B_i$'s, going forward, we can assume that it is one 
of $B_1$, $B_4$, and $B_6$ that is linked. Recall that our ultimate
goal is to argue that none of the $B_i$ are linked and thereby force
a contradiction.

\smallskip
\noindent%
{\bf Step VI:} Eliminate $C_1$ and $A_7$. This leaves only $A_1$ and $A_6$
among the $A_i$ pairs.

Our next step is to argue that $C_1$ is not linked by 
comparing with the remaining three $A_i$'s.
For $A_1$ split the second cycle of $C_1$
$\mbox{0,2,9,5,1,7 } = \mbox{0,2,9,5 } \cup \mbox{ 0,5,1,7}$,
yielding first a $\DD$ on 
$\{0,5\}, \{2,9\}, \{3\}, \{4,10\}$ and then on
$\{0,5\}, \{1,7\}, \{3\}, \{4,10\}$.
For $A_6$, we have $\DD$ with vertices
$\{0,2,7,9\}, \{1,5\},$ $ \{3\}, \{10\}$.

For $A_7$, split the second cycle
$\mbox{0,2,6,1,7  } =  \mbox{0,2,8,1,7 } \cup \mbox{ 1,6,2,8}$.
In the first case, we have a $\DD$ with vertices
$\{0,1,2,7\}, \{3,10\}, \{5,9\}, \{8\}$. In
the second case, split the second cycle of $C_1$
$\mbox{0,2,9,5,1,7 } = \mbox{0,2,9,5 } \cup \mbox{ 0,5,1,7}$,
resulting in a $\DD$ either on 
$\{2\}, \{3,10\}, \{5,9\}, \{8\}$ or
$\{1\}, \{3,10\}, \{5\}, \{8\}$.

Now we can eliminate $A_7$. 
Using a symmetry of $G_2$ we will instead argue that the
pair $A_7'$ = 3,8,4,10 - 0,2,6,1,7 is not linked.
Note that this resembles $C_1$, which we just proved unlinked.
Since $\mbox{0,2,6,1,7 } \cup \mbox{ 0,2,9,5,1,7 } = \mbox{1,5,9,2,6}$
it will be enough to show that 3,8,4,10 - 1,5,9,2,6 
is not linked by comparing with the remaining $A_i$'s:
\begin{align*}
A_1 & \{1,2,6,9\}, \{3\}, \{4,10\}, \{5\} &
A_6 & \{1,5,6\}, \{2,9\}, \{3\}, \{10\} \\
A_7 & \{0,4\}, \{1,2,6\}, \{3,10\}, \{5,9\} \\
\end{align*}
Thus we can assume that it is $A_1$ or $A_6$ that is the linked
pair in our embedding of $G_2$.

\smallskip
\noindent%
{\bf Step VII:} Eliminate $B_1$ and $B_6$, leaving only $B_4$.

As for the $B_i$'s, only three candidates remain. We next eliminate
$B_1$ by comparing with the remaining two $A_i$'s.
For $A_1$, split the second cycle of $B_1$:
$\mbox{1,5,9,3,10,4,8 } = \mbox{1,5,10,4,8 } \cup \mbox{ 3,9,10,5}$
giving a $\DD$ on either
$\{0\}, \{1\}, \{2,6\}, \{4,5,10\}$
or $\{0\}, \{2,6\}, \{3,9\}, \{5,10\}$.
For $A_6$, using the same split of the second cycle of $B_1$
the vertices are either
$\{0,2\}, \{1,5,10\},$ $ \{4,9\}, \{6\}$ or
$\{0,2\}, \{3,9\},  \{5,10\}, \{6\}$.

To proceed, we will argue that $C_5$ is unlinked. In fact, we will 
show that it is $D_6 = C_5'$, the result of applying the symmetry 
of $G_2$, that is not linked by comparing with the two remaining 
$A_i$'s:
\begin{align*}
A_1& \{0,5\}, \{1,2,6\}, \{3\}, \{4,10\} &
A_6& \{0,2\}, \{1,5,6\}, \{3\}, \{10\}
\end{align*}

We can now eliminate $B_6$ by comparing with the $C_i$'s.
This will leave only $B_4$, which, therefore, must be linked.
We have already argued that $C_1$ and $C_5$ are not linked. 
Also, by a symmetry of $G_2$, since $B_7$ is not linked, $C_4$ is also
not linked. For $C_3$ and $C_7$, we immediately see a $\DD$:
\begin{align*}
C_3 & \{0,4,5,10\}, \{1\}, \{2,6\}, \{3,9\} &
C_7 & \{0,4\}, \{1\}, \{2,8\}, \{3,9,5\}.
\end{align*}

For $C_2$, split the second cycle:
$\mbox{0,2,9,5,10,4 } = \mbox{0,2,6,10,4 } \cup \mbox{ 2,6,10,5,9}$.
In the first case, we have a $\DD$ on 
$\{0,4,10\}, \{1,8\}, \{2,6\}, \{3\}$.
In the second case, split the second cycle of $B_6$:
$\mbox{0,4,10,3,9,5 } = \mbox{0,4,10,3,7 } \cup \mbox{ 0,5,9,3,7}$
giving a $\DD$ with vertices
$\{1,8\}, \{2,6\}, \{3,7\}, \{10\}$ in the first subcase
and $\{1,8\}, \{2,6\}, \{3,7\}, \{5,9\}$ in the second.

For $C_6$, we split the second cycle of $B_6$:
$\mbox{0,4,10,3,9,5 } = \mbox{0,4,10,5} \cup \mbox{ 3,9,5,10}$
giving, first, a $\DD$ on $\{0\}, \{1,8\}, \{2\}, \{4,5,10\}$ 
and, second, a $\DD$ on $\{1,8\}, \{2\},$ $ \{3,9\}, \{5,10\}$.

\smallskip
\noindent%
{\bf Step VIII:} Introduce $F_i$ pairs to eliminate $B_4$ and complete the argument.

Since $B_1, \ldots, B_7$, represent the pairs of cycles
in an embedding of $P_9$, we know by \cite{S} that at least
one pair must have odd linking number. We have just argued 
that all but $B_4$ are not linked, so we can conclude that
it is $B_4$ that has odd linking number in our embedding 
of $G_2$. We will now derive a contradiction by using a final
Petersen family graph minor.

\begin{table}[htb]
    \centering
    \begin{tabular}{c|l}
         $F_1$ & 0,5,1,7 -- 2,6,10,3,8 \\
         $F_2$ & 0,4,8,1,5 -- 2,6,10,3,7 \\
         $F_3$ & 0,2,8,4 -- 1,6,10,3,7 \\
         $F_4$ & 0,2,7 -- 1,6,10,3,8 \\
         $F_5$ & 1,5,10,6 -- 2,7,3,8 \\
         $F_6$ & 1,7,3,8 -- 0,2,6,10,5 \\
         $F_7$ & 1,6,2,8 -- 0,5,10,3,7 \\
         $F_8$ & 1,6,2,7 -- 0,4,8,3,10,5 \\
    \end{tabular}
    \caption{Eight pairs of cycles in $G_2$ called $F_1, \ldots, F_8$.}
    \label{tab:G2Fn}
\end{table}

Our last set of cycles comes from a $P_8$ minor. This is the
Petersen family graph on eight vertices that is not $K_{4,4}^-$.
Using the labelling of $G_2$ in Figure~\ref{fig:G2}, 
contracting edges $(0,4)$ and $(0,5)$ and deleting vertex 9
results in a graph with a $P_8$ subgraph. This graph has
eight pairs of cycles shown in Table~\ref{tab:G2Fn}.

Using $B_4$ will derive a $\DD$ with each $F_i$. 
For the first five $F_i$ we immediately find a $\DD$:

\begin{align*}
F_1 & \{0\}, \{1,2,6\}, \{3\}, \{4,10\} &
F_2 & \{0,2\}, \{1,5,6\}, \{3\}, \{10\} \\
F_3 & \{0,5\}, \{1,2,6\}, \{3\}, \{4,10\} &
F_4 & \{0,2\}, \{1,5,6\}, \{3\}, \{10\} \\
F_5 & \{0,5\}, \{1,2,6\}, \{3\}, \{4,10\}. \\
\end{align*}

Since $B_4$ is linked, we deduce that the pair 
$B_4' = 2,7,3,9 -- 0,4,8,1,5$, obtained by the 
symmetry of $G_2$, is also linked. Using $B_4'$
we have a $\DD$ with the remaining cycles of $F$:

\begin{align*}
F_6 & \{0,5\}, \{1,8\}, \{2\}, \{3,7\} &
F_7 & \{0,5\}, \{1,8\}, \{2\}, \{3,7\} \\
F_8 & \{0,4,5,8\}, \{1\}, \{2,7\}, \{3\} \\
\end{align*}

We have shown that there is a $\DD$ with each $F_i$ using
the pairs $B_4$ or $B_4'$, both of which must be linked. 
Since the $F_1, \ldots, F_8$
represent the cycle pairs of a $P_8$ minor, at least one
of them has odd linking number~\cite{S}. By 
Lemma~\ref{lem:D4}, our embedding of $G_2$ has a
knotted cycle. This contradicts our assumption
that we were working with a knotless embedding. 
The contradiction shows that there is no such knotless
embedding and $G_2$ is IK. This completes the proof
that $G_2$ is MMIK.

\section{Maxnik graphs of order ten}
\label{sec:appmnik}


In this section we describe the maxnik graphs of order ten. Recall (see~\cite{EFM}) that a maximal $2$-apex graph is maxnik. 
There are 14 maximal $2$-apex graphs formed as the join $K_2 \ast T_8$ of $K_2$ with one of the 14 triangulations
on eight vertices, see Bowen and Fisk~\cite{BF}. 
We list the 14 maximal $2$-apex graphs of order ten in Appendix C.
Aside from those 14 there are 35 additional maxnik graphs. Our computer search, described in 
Section~\ref{sec:ord10} above, shows that there are no other maxnik
graphs beyond these 49. In this section we argue that the 35
non $2$-apex graphs that
we have found are indeed maxnik. The graphs are listed below.

To show a graph is maxnik requires two things. Using Naimi's implementation~\cite{NW}
of Miller and Naimi's~\cite{MN} algorithm, we have verified for each graph $G$
that whenever we add
an edge $e \not\in E(G)$, the graph $G+e$ is IK. It remains to show that each of the graphs is nIK. We divide the graphs into three bins. For the first
14 graphs we argue the graph is nIK by demonstrating a $2$-apex child. We handle the next three graphs using two lemmas from \cite{EFM}. For the 
remaining 18 graphs, we give a knotless embedding.

\subsection{Graphs with a $2$-apex child}
In this subsection we list the first 14 of the 35 maxnik graphs of order ten  that are not $2$-apex. We show these graphs are nIK by 
demonstrating a $2$-apex child.

For each graph $G$ in this list of 14, 
we provide the size, graph6 format~\cite{sage}, edge list, and the
vertices of a 
triangle in the graph. Making a $\ty$ move on that triangle 
results in a child $H$ that is $2$-apex. In each case, $H$ becomes planar
on deleting vertex $9$ and the new degree $3$ vertex. 
By Lemma~\ref{lem:tyyt}, since $G$ has a $2$-apex (hence nIK) child, 
$G$ is also nIK. 

\begin{enumerate}
    \item Size: 33; graph6 format: \verb"ICf^f\~~w"; triangle: $0,3,6$
$$[(0, 3), (0, 4), (0, 5), (0, 6), (0, 7), (0, 9), (1, 5), (1, 6), (1, 7), (1, 8), (1, 9),$$ 
$$(2, 6), (2, 7), (2, 8), (2, 9), (3, 4), (3, 5), (3, 6), (3, 8), (3, 9), (4, 5), (4, 7),$$
$$(4, 8), (4, 9), (5, 7), (5, 8), (5, 9), (6, 7), (6, 8), (6, 9), (7, 8), (7, 9), (8, 9)]$$
    \item Size: 33; graph6 format: \verb"ICxu|~{~w"; triangle: $0,3,8$
$$[(0, 3), (0, 4), (0, 6), (0, 7), (0, 8), (0, 9), (1, 4), (1, 5), (1, 6), (1, 8), (1, 9),$$
$$(2, 4), (2, 5), (2, 7), (2, 8), (2, 9), (3, 5), (3, 6), (3, 7), (3, 8), (3, 9), (4, 6),$$
$$(4, 7), (4, 8), (4, 9), (5, 6), (5, 7), (5, 8), (5, 9), (6, 7), (6, 9), (7, 9), (8, 9)]$$
    \item Size: 33; graph6 format: \verb"ICvbm~}~w"; triangle $0,5,7$
$$[(0, 3), (0, 4), (0, 5), (0, 7), (0, 8), (0, 9), (1, 4), (1, 5), (1, 6), (1, 7), (1, 8),$$ 
$$(1, 9), (2, 5), (2, 6), (2, 8), (2, 9), (3, 4), (3, 6), (3, 7), (3, 8), (3, 9), (4, 7),$$ 
$$(4, 8), (4, 9), (5, 6), (5, 7), (5, 8), (5, 9), (6, 7), (6, 8), (6, 9), (7, 9), (8, 9)]$$
    \item Size: 33; graph6 format: \verb"IEirt~}~w"; triangle $0,4,7$
$$[(0, 3), (0, 4), (0, 5), (0, 7), (0, 8), (0, 9), (1, 3), (1, 6), (1, 8), (1, 9), (2, 4),$$ 
$$(2, 5), (2, 6), (2, 7), (2, 8), (2, 9), (3, 5), (3, 6), (3, 7), (3, 8), (3, 9), (4, 6),$$ 
$$(4, 7), (4, 8), (4, 9), (5, 7), (5, 8), (5, 9), (6, 7), (6, 8), (6, 9), (7, 9), (8, 9)]$$
    \item Size: 33; graph6 format: \verb"IEhuV~}~w"; triangle $0,3,7$
$$[(0, 3), (0, 4), (0, 6), (0, 7), (0, 8), (0, 9), (1, 3), (1, 5), (1, 6), (1, 7), (1, 8),$$ 
$$(1, 9), (2, 4), (2, 5), (2, 7), (2, 8), (2, 9), (3, 5), (3, 7), (3, 8), (3, 9), (4, 6),$$ 
$$(4, 7), (4, 8), (4, 9), (5, 7), (5, 8), (5, 9), (6, 7), (6, 8), (6, 9), (7, 9), (8, 9)]$$
    \item Size: 33; graph6 format: \verb"IEhvVn}~w"; triangle $0,3,7$
$$[(0, 3), (0, 4), (0, 6), (0, 7), (0, 8), (0, 9), (1, 3), (1, 5), (1, 6), (1, 7), (1, 8),$$ 
$$(1, 9), (2, 4), (2, 5), (2, 6), (2, 7), (2, 8), (2, 9), (3, 5), (3, 7), (3, 8), (3, 9),$$ 
$$(4, 6), (4, 8), (4, 9), (5, 7), (5, 8), (5, 9), (6, 7), (6, 8), (6, 9), (7, 9), (8, 9)]$$
    \item Size: 33; graph6 format: \verb"IEh~f]}~w"; triangle $0,4,7$
$$[(0, 3), (0, 4), (0, 6), (0, 7), (0, 8), (0, 9), (1, 3), (1, 5), (1, 6), (1, 7), (1, 9),$$ 
$$(2, 4), (2, 5), (2, 6), (2, 7), (2, 8), (2, 9), (3, 5), (3, 6), (3, 8), (3, 9), (4, 5),$$
$$(4, 7), (4, 8), (4, 9), (5, 7), (5, 8), (5, 9), (6, 7), (6, 8), (6, 9), (7, 9), (8, 9)]$$
    \item Size: 33; graph6 format: \verb"IQjnex~~w"; triangle $0,2,6$
$$[(0, 2), (0, 4), (0, 5), (0, 6), (0, 7), (0, 9), (1, 3), (1, 5), (1, 6), (1, 7), (1, 8),$$ 
$$(1, 9), (2, 4), (2, 5), (2, 6), (2, 8), (2, 9), (3, 6), (3, 7), (3, 8), (3, 9), (4, 5),$$ 
$$(4, 7), (4, 8), (4, 9), (5, 7), (5, 8), (5, 9), (6, 8), (6, 9), (7, 8), (7, 9), (8, 9)]$$
    \item Size: 33; graph6 format: \verb"IQzTuz}~w"; triangle $3,6,8$
$$[(0, 2), (0, 4), (0, 5), (0, 6), (0, 7), (0, 8), (0, 9), (1, 3), (1, 4), (1, 5), (1, 7),$$
$$(1, 8), (1, 9), (2, 4), (2, 6), (2, 8), (2, 9), (3, 5), (3, 6), (3, 7), (3, 8), (3, 9),$$
$$(4, 6), (4, 7), (4, 8), (4, 9), (5, 7), (5, 8), (5, 9), (6, 8), (6, 9), (7, 9), (8, 9)]$$
    \item Size: 33; graph6 format: \verb"IQzTvz]~w"; triangle $3,6,8$
$$[(0, 2), (0, 4), (0, 5), (0, 6), (0, 7), (0, 8), (0, 9), (1, 3), (1, 4), (1, 5), (1, 7),$$ 
$$(1, 8), (1, 9), (2, 4), (2, 6), (2, 7), (2, 9), (3, 5), (3, 6), (3, 7), (3, 8), (3, 9),$$
$$(4, 6), (4, 7), (4, 8), (4, 9), (5, 7), (5, 8), (5, 9), (6, 8), (6, 9), (7, 9), (8, 9)]$$
    \item Size: 33; graph6 format: \verb"IQzTu~]~w"; triangle $0,5,7$
$$[(0, 2), (0, 4), (0, 5), (0, 6), (0, 7), (0, 8), (0, 9), (1, 3), (1, 4), (1, 5), (1, 7),$$ 
$$(1, 8), (1, 9), (2, 4), (2, 6), (2, 9), (3, 5), (3, 6), (3, 7), (3, 8), (3, 9), (4, 6),$$ 
$$(4, 7), (4, 8), (4, 9), (5, 7), (5, 8), (5, 9), (6, 7), (6, 8), (6, 9), (7, 9), (8, 9)]$$
    \item Size: 33; graph6 format: \verb"IQyuvx}~w"; triangle $0,4,6$
$$[(0, 2), (0, 4), (0, 5), (0, 6), (0, 7), (0, 9), (1, 3), (1, 4), (1, 6), (1, 7), (1, 8),$$ 
$$(1, 9), (2, 4), (2, 5), (2, 7), (2, 8), (2, 9), (3, 5), (3, 6), (3, 7), (3, 8), (3, 9),$$ 
$$(4, 6), (4, 7), (4, 8), (4, 9), (5, 7), (5, 8), (5, 9), (6, 8), (6, 9), (7, 9), (8, 9)]$$
    \item Size: 33; graph6 format: \verb"IQyvux|~w"; triangle $1,3,6$
$$[(0, 2), (0, 4), (0, 5), (0, 6), (0, 7), (0, 9), (1, 3), (1, 4), (1, 6), (1, 7), (1, 8),$$
$$(1, 9), (2, 4), (2, 5), (2, 6), (2, 8), (2, 9), (3, 5), (3, 6), (3, 7), (3, 8), (3, 9),$$ 
$$(4, 6), (4, 7), (4, 8), (4, 9), (5, 7), (5, 8), (5, 9), (6, 9), (7, 8), (7, 9), (8, 9)]$$
    \item Size: 33; graph6 format: \verb"IUZurzm~w"; triangle $0,3,6$
$$[(0, 2), (0, 3), (0, 5), (0, 6), (0, 8), (0, 9), (1, 3), (1, 4), (1, 5), (1, 6), (1, 7),$$ 
$$(1, 8), (1, 9), (2, 4), (2, 5), (2, 7), (2, 8), (2, 9), (3, 5), (3, 6), (3, 7), (3, 9),$$ 
$$(4, 6), (4, 7), (4, 8), (4, 9), (5, 7), (5, 8), (5, 9), (6, 8), (6, 9), (7, 9), (8, 9)]$$
\end{enumerate}

\subsection{Lemmas from \cite{EFM}}
In this subsection, we argue that graphs 15, 16, and 17 (listed below) are nIK using
ideas from \cite{EFM}. 

Graphs 16 and 17 have a degree $2$ vertex and
we can recognize them as clique sums. In both cases, the graph is the sum
of $K_3$ and the Heawood family graph $E_9$ over $K_2$. Since $K_3$
and $E_9$ are both nIK, by \cite[Lemma 3.1]{EFM} graphs 16 and 17 are nIK.

Graph 15 has a degree $3$ vertex and is the clique sum over $K_3$ of $K_4$
and $E_9$. In the embedding of $E_9$ in \cite[Figure 2]{EFM}
the $K_3$ is given by the vertices $a,b,c$, which bound a disk
whose interior is disjoint from the graph. By \cite[Lemma 3.4]{EFM},
Graph 15 is also nIK.

\begin{enumerate}
\setcounter{enumi}{14}

    \item Size: 24; graph6 format: \verb"ICRffQmn_"
$$[(0, 3), (0, 5), (0, 6), (0, 7), (0, 8), (0, 9), (1, 4), (1, 5), (1, 6), (1, 7), (2, 5), (2, 6),$$ $$(2, 7), (2, 8), (2, 9), (3, 6), (3, 9), (4, 7), (4, 8), (4, 9), (5, 8), (5, 9), (6, 8), (6, 9)]$$
    \item Size: 23; graph6 format: \verb"I?qtdo}^_"
$$[(0, 4), (0, 5), (0, 6), (0, 7), (1, 4), (1, 9), (2, 5), (2, 6), (2, 7), (2, 8), (2, 9), (3, 5),$$
$$(3, 6), (3, 7), (3, 8), (3, 9), (4, 7), (4, 8), (4, 9), (5, 8), (5, 9), (6, 8), (6, 9)]$$
    \item Size: 23; graph6 format: \verb"I?qtfo}N_"
$$[(0, 4), (0, 5), (0, 6), (0, 7), (1, 4), (1, 7), (2, 5), (2, 6), (2, 7), (2, 8), (2, 9), (3, 5),$$ $$(3, 6), (3, 7), (3, 8), (3, 9), (4, 7), (4, 8), (4, 9), (5, 8), (5, 9), (6, 8), (6, 9)]$$
\end{enumerate}

\subsection{Knotless embeddings}

\begin{figure}[htb]
\centering
\includegraphics[scale=.8]{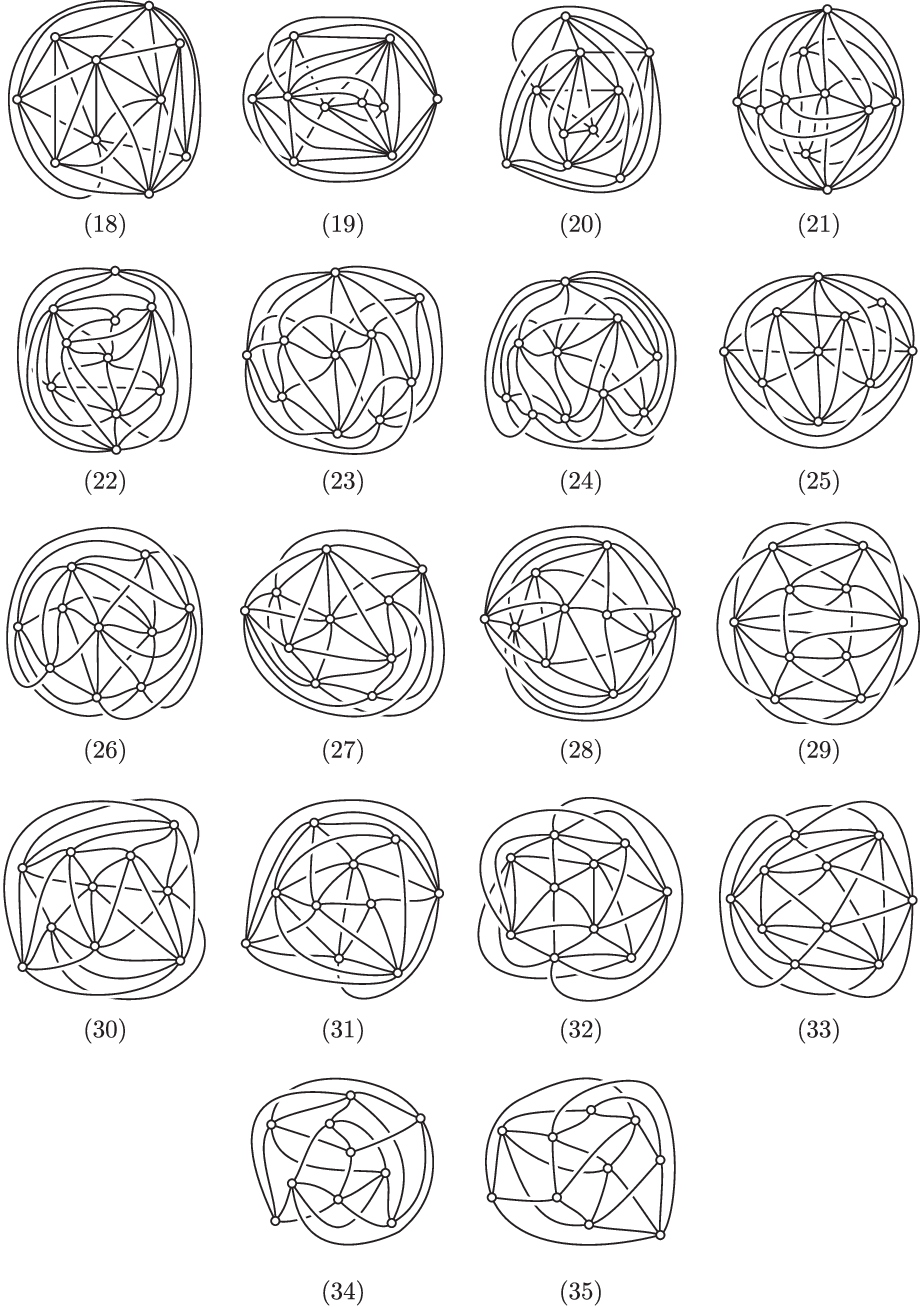}
\caption{Knotless embeddings of 18 graphs.} 
\label{fig:18nik}
\end{figure}

We show that the remaining 18 graphs (listed below) are nIK by 
presenting knotless
embeddings in Figure~\ref{fig:18nik}. Recall that this means when 
we apply Naimi's findEasyKnots program~\cite{NW}
to the embedding, it confirms that every cycle in the graph
is an unknot.

\begin{enumerate}
\setcounter{enumi}{17}
    \item Size: 34; graph6 format: \verb"IQjuz~nnw"
$$[(0, 2), (0, 4), (0, 5), (0, 6), (0, 8), (0, 9), (1, 3), (1, 5), (1, 6), (1, 7), (1, 8), (2, 4),$$ $$(2, 5), (2, 7), (2, 8), (2, 9), (3, 5), (3, 6), (3, 7), (3, 9), (4, 6), (4, 7), (4, 8),$$
$$(4, 9), (5, 6), (5, 7), (5, 8), (5, 9), (6, 7), (6, 8), (6, 9), (7, 8), (7, 9), (8, 9)]$$
    \item Size: 33; graph6 format: \verb"IQjUnzz~w"
$$[(0, 2), (0, 4), (0, 5), (0, 6), (0, 7), (0, 8), (0, 9), (1, 3), (1, 5), (1, 6), (1, 7),$$ 
$$(1, 8), (1, 9), (2, 4), (2, 7), (2, 8), (2, 9), (3, 5), (3, 6), (3, 7), (3, 8), (3, 9),$$ 
$$(4, 7), (4, 8), (4, 9), (5, 6), (5, 7), (5, 9), (6, 8), (6, 9), (7, 8), (7, 9), (8, 9)]$$
    \item Size: 33; graph6 format: \verb"IQjne~^^w"
$$[(0, 2), (0, 4), (0, 5), (0, 6), (0, 7), (0, 8), (1, 3), (1, 5), (1, 6), (1, 7), (1, 8),$$ 
$$(1, 9), (2, 4), (2, 5), (2, 6), (2, 9), (3, 6), (3, 7), (3, 8), (3, 9), (4, 5), (4, 7),$$ 
$$(4, 8), (4, 9), (5, 7), (5, 8), (5, 9), (6, 7), (6, 8), (6, 9), (7, 8), (7, 9), (8, 9)]$$
    \item Size: 33; graph6 format: \verb"IQjne|~~W"
$$[(0, 2), (0, 4), (0, 5), (0, 6), (0, 7), (0, 9), (1, 3), (1, 5), (1, 6), (1, 7), (1, 8),$$ 
$$(1, 9), (2, 4), (2, 5), (2, 6), (2, 8), (2, 9), (3, 6), (3, 7), (3, 8), (3, 9), (4, 5),$$ 
$$(4, 7), (4, 8), (4, 9), (5, 7), (5, 8), (5, 9), (6, 7), (6, 8), (7, 8), (7, 9), (8, 9)]$$
    \item Size: 33; graph6 format: \verb"IQjne|~zw"
$$[(0, 2), (0, 4), (0, 5), (0, 6), (0, 7), (0, 9), (1, 3), (1, 5), (1, 6), (1, 7), (1, 8),$$ 
$$(1, 9), (2, 4), (2, 5), (2, 6), (2, 8), (2, 9), (3, 6), (3, 7), (3, 8), (4, 5), (4, 7),$$ 
$$(4, 8), (4, 9), (5, 7), (5, 8), (5, 9), (6, 7), (6, 8), (6, 9), (7, 8), (7, 9), (8, 9)]$$
    \item Size: 33; graph6 format: \verb"IQzTvz^~o"
$$[(0, 2), (0, 4), (0, 5), (0, 6), (0, 7), (0, 8), (0, 9), (1, 3), (1, 4), (1, 5), (1, 7),$$ 
$$(1, 8), (1, 9), (2, 4), (2, 6), (2, 7), (2, 9), (3, 5), (3, 6), (3, 7), (3, 8), (3, 9),$$ 
$$(4, 6), (4, 7), (4, 8), (4, 9), (5, 7), (5, 8), (5, 9), (6, 8), (6, 9), (7, 8), (7, 9)]$$
    \item Size: 33; graph6 format: \verb"IQzTvv^~o"
$$[(0, 2), (0, 4), (0, 5), (0, 6), (0, 7), (0, 8), (0, 9), (1, 3), (1, 4), (1, 5), (1, 7),$$ 
$$(1, 8), (1, 9), (2, 4), (2, 6), (2, 7), (2, 9), (3, 5), (3, 6), (3, 7), (3, 8), (3, 9),$$ 
$$(4, 6), (4, 7), (4, 8), (4, 9), (5, 8), (5, 9), (6, 7), (6, 8), (6, 9), (7, 8), (7, 9)]$$
    \item Size: 33; graph6 format: \verb"IQzTu~^~o"
$$[(0, 2), (0, 4), (0, 5), (0, 6), (0, 7), (0, 8), (0, 9), (1, 3), (1, 4), (1, 5), (1, 7),$$ 
$$(1, 8), (1, 9), (2, 4), (2, 6), (2, 9), (3, 5), (3, 6), (3, 7), (3, 8), (3, 9), (4, 6),$$ 
$$(4, 7), (4, 8), (4, 9), (5, 7), (5, 8), (5, 9), (6, 7), (6, 8), (6, 9), (7, 8), (7, 9)]$$
    \item Size: 33; graph6 format: \verb"IQyuz~{~o"
$$[(0, 2), (0, 4), (0, 5), (0, 6), (0, 8), (0, 9), (1, 3), (1, 4), (1, 6), (1, 7), (1, 8),$$ 
$$(1, 9), (2, 4), (2, 5), (2, 7), (2, 8), (2, 9), (3, 5), (3, 6), (3, 7), (3, 8), (3, 9),$$ 
$$(4, 6), (4, 7), (4, 8), (4, 9), (5, 6), (5, 7), (5, 8), (5, 9), (6, 7), (6, 9), (7, 9)]$$
    \item Size: 33; graph6 format: \verb"IQyuz~{zw"
$$[(0, 2), (0, 4), (0, 5), (0, 6), (0, 8), (0, 9), (1, 3), (1, 4), (1, 6), (1, 7), (1, 8),$$ 
$$(1, 9), (2, 4), (2, 5), (2, 7), (2, 8), (2, 9), (3, 5), (3, 6), (3, 7), (3, 8), (4, 6),$$ 
$$(4, 7), (4, 8), (4, 9), (5, 6), (5, 7), (5, 8), (5, 9), (6, 7), (6, 9), (7, 9), (8, 9)]$$
    \item Size: 33; graph6 format: \verb"IQyuz~{vw"
$$[(0, 2), (0, 4), (0, 5), (0, 6), (0, 8), (0, 9), (1, 3), (1, 4), (1, 6), (1, 7), (1, 8),$$ 
$$(1, 9), (2, 4), (2, 5), (2, 7), (2, 8), (3, 5), (3, 6), (3, 7), (3, 8), (3, 9), (4, 6),$$ 
$$(4, 7), (4, 8), (4, 9), (5, 6), (5, 7), (5, 8), (5, 9), (6, 7), (6, 9), (7, 9), (8, 9)]$$
    \item Size: 32; graph6 format: \verb"IQzTrj~~o"
$$[(0, 2), (0, 4), (0, 5), (0, 6), (0, 8), (0, 9), (1, 3), (1, 4), (1, 5), (1, 7), (1, 8),$$ 
$$(1, 9), (2, 4), (2, 6), (2, 7), (2, 8), (2, 9), (3, 5), (3, 6), (3, 7), (3, 8), (3, 9),$$ 
$$(4, 6), (4, 8), (4, 9), (5, 7), (5, 8), (5, 9), (6, 8), (6, 9), (7, 8), (7, 9)]$$
    \item Size: 32; graph6 format: \verb"IUZuvzmno"
$$[(0, 2), (0, 3), (0, 5), (0, 6), (0, 7), (0, 8), (0, 9), (1, 3), (1, 4), (1, 5), (1, 6),$$ 
$$(1, 7), (1, 8), (2, 4), (2, 5), (2, 7), (2, 8), (2, 9), (3, 5), (3, 6), (3, 7), (3, 9),$$
$$ (4, 6), (4, 7), (4, 8), (4, 9), (5, 7), (5, 8), (5, 9), (6, 8), (6, 9), (7, 9)]$$
    \item Size: 31; graph6 format: \verb"IEivux~zo"
$$[(0, 3), (0, 4), (0, 5), (0, 6), (0, 7), (0, 9), (1, 3), (1, 6), (1, 7), (1, 8), (1, 9),$$ 
$$(2, 4), (2, 5), (2, 6), (2, 8), (2, 9), (3, 5), (3, 6), (3, 7), (3, 8), (4, 6), (4, 7),$$ 
$$(4, 8), (4, 9), (5, 7), (5, 8), (5, 9), (6, 8), (6, 9), (7, 8), (7, 9)]$$
    \item Size: 31; graph6 format: \verb"IEhvuzn^o"
$$[(0, 3), (0, 4), (0, 6), (0, 7), (0, 8), (1, 3), (1, 5), (1, 6), (1, 7), (1, 8), (1, 9), $$
$$(2, 4), (2, 5), (2, 6), (2, 8), (2, 9), (3, 5), (3, 6), (3, 7), (3, 9), (4, 6), (4, 7), $$
$$(4, 8), (4, 9), (5, 7), (5, 8), (5, 9), (6, 8), (6, 9), (7, 8), (7, 9)]$$
    \item Size: 31; graph6 format: \verb"IEnb~jm}W"
$$[(0, 3), (0, 4), (0, 5), (0, 7), (0, 8), (0, 9), (1, 3), (1, 5), (1, 6), (1, 7), (1, 8), $$
$$(1, 9), (2, 4), (2, 5), (2, 6), (2, 7), (2, 8), (2, 9), (3, 4), (3, 6), (3, 7), (3, 9), $$
$$(4, 6), (4, 8), (4, 9), (5, 6), (5, 7), (5, 8), (6, 8), (7, 9), (8, 9)]$$
    \item Size: 23; graph6 format: \verb"ICpVbrkN_"
$$[(0, 3), (0, 4), (0, 6), (0, 8), (1, 4), (1, 5), (1, 6), (1, 7), (1, 8), (2, 6), (2, 7), (2, 8),$$ 
$$(2, 9), (3, 5), (3, 6), (3, 7), (3, 9), (4, 7), (4, 8), (4, 9), (5, 8), (5, 9), (6, 9)]$$
    \item Size: 23; graph6 format: \verb"ICpvbqkN_"
$$[(0, 3), (0, 4), (0, 6), (0, 8), (1, 4), (1, 5), (1, 6), (1, 7), (2, 5), (2, 6), (2, 7), (2, 8),$$ 
$$(2, 9), (3, 5), (3, 6), (3, 7), (3, 9), (4, 7), (4, 8), (4, 9), (5, 8), (5, 9), (6, 9)]$$
\end{enumerate}

\end{document}